\documentclass[11pt]{article}
\usepackage{amsfonts}
\usepackage{graphics}
\usepackage{indentfirst}
\usepackage{color}
\usepackage{cite}
\usepackage{latexsym}
\usepackage[paper=a4paper, left=1.6cm, right=1.6cm, top=1.8cm, bottom=1.6cm, headheight=5.5pt, footskip=0.8cm, footnotesep=0.8cm, centering, includefoot]{geometry}
\usepackage{amsmath}
\allowdisplaybreaks
\usepackage{amssymb}
\usepackage[colorlinks, linkcolor=red]{hyperref}
\hypersetup{urlcolor=red, citecolor=red}
\usepackage[dvips]{epsfig}
\usepackage{amscd}

\usepackage{amsthm}
\usepackage{mathrsfs}
\usepackage{verbatim}
\newtheorem{theorem}{Theorem}[section]
\newtheorem{remark}{Remark}[section]
\newtheorem{defn}{Definition}[section]
\newtheorem{lemma}{Lemma}[section]

\newtheorem{proposition}{Proposition}[section]

\allowdisplaybreaks

\let\pa=\partial
\let\f=\frac

\def\na{\nabla}

\renewcommand{\div}{{\rm div}}

\DeclareMathOperator{\divv}{div}

\DeclareMathOperator{\curl}{curl}

\makeatletter
\@addtoreset{equation}{section}
\makeatother
\makeatletter
\@addtoreset{equation}{section}
\makeatother

\title{Global weak solutions of 3D compressible magnetohydrodynamic equations subject to large external potential forces with discontinuous initial data and vacuum
\thanks{Chen's research was partially supported by Postgraduate Research and Innovation Project of Southwest University (No. SWUS25083). Zhong's research was partially supported by National Natural Science Foundation of China (No. 12371227) and Fundamental Research Funds for the Central Universities (No. SWU--KU24001).}
}

\author{Geyuan Chen,\ Xin Zhong
{\thanks{E-mail addresses: gychen17@outlook.com (G. Chen),  xzhong1014@amss.ac.cn (X. Zhong).}}
\date{}\\
\footnotesize School of Mathematics and Statistics, Southwest University, Chongqing 400715, P. R. China}

\begin{document}
\maketitle

\begin{abstract}
We investigate the compressible magnetohydrodynamic equations subject to large external potential forces with discontinuous initial data in a three-dimensional bounded domain under Navier-slip boundary conditions. We show the global existence of weak solutions for such an initial-boundary value problem provided the initial energy is suitably small. In particular, the initial data may contain vacuum states and possibly exhibit large oscillations. To overcome difficulties brought by boundary and large external forces, some new estimates based on the effective viscous flux play crucial roles.
\end{abstract}




\section{Introduction}
\par We consider a viscous isentropic compressible magnetohydrodynamic flow in a bounded domain $\Omega\subset\mathbb{R}^3$:
\begin{align}\label{a1}
\begin{cases}
\rho_t+\divv (\rho u)=0,\\
(\rho u)_t+\divv (\rho u\otimes u)-\mu\Delta u-(\mu+\lambda)\nabla\divv u+\nabla P=(\nabla\times H)\times H+\rho f,\\
H_{t}-\nabla\times(u\times H)+\nu\nabla\times(\nabla\times H)=0,\\
\divv H=0,
\end{cases}
\end{align}
where the unknowns $\rho$, $u=(u^1, u^2, u^3)$, $P=A\rho^{\gamma}\ (A>0, \gamma>1)$, and $  H=(H^1,H^2,H^3)$ stand for the density, velocity, pressure, and magnetic field, respectively. The constants $\mu$ and $\lambda$ represent the shear viscosity and the bulk viscosity, respectively, satisfying the physical restrictions
\begin{align*}
\mu>0,\quad 2\mu+3\lambda\ge 0,
\end{align*}
while the constant $\nu>0$ is the resistivity coefficient. We assume that the external force $f$ is a gradient of a scalar potential $\phi=\phi(x)$ (i.e., $f=\nabla\phi(x)$).

The system \eqref{a1} is supplemented with the initial data
\begin{align}
(\rho, \rho u, H)(x, 0)=(\rho_{0}, \rho_0u_{0}, H_{0})(x), \quad x\in \Omega,
\end{align}
and Navier-slip boundary conditions
\begin{gather}\label{a5}
u\cdot n=0, \ \curl u\times n=0,\,
H\cdot n=0, \ \curl H\times n=0,\quad x\in\partial\Omega,\ t>0,
\end{gather}
where $n$ is the unit outward normal vector to $\pa\Omega$.

Magnetohydrodynamics (MHD) is the study of the dynamics of conductive fluids
such as plasmas in the presence of electromagnetic fields, which was established in the 1940s by H. Alfv{\'e}n \cite{A41} and was successfully applied to geophysics and astrophysics (see, e.g., \cite{D17}). Due to its significant physical relevance, complexity, diverse phenomena, and mathematical difficulties, there have been numerous studies on the global well-posedness of solutions for the compressible MHD system. We refer the reader to \cite[Chapter 3]{LT12}, which provides a detailed derivation of \eqref{a1} from the general constitutive laws.

Recently significant progress has been made concerning the global existence of solutions for the isentropic compressible MHD equations \eqref{a1} in the absence of external force. Following the classical approach introduced by Matsumura and Nishida \cite{MN80,MN83} in the 1980s for compressible viscous flows, Li--Yu \cite{LY11} and Zhang--Zhao \cite{ZZ10} independently established the global unique classical solution when the initial data is a small perturbation of a stable reference state $(\rho_\infty,0,0)$ in $H^3(\mathbb{R}^3)$ with constant positive density $\rho_\infty$.
It should be pointed out that the initial data is close to a constant state with a small $H^3$-norm, which indicates the absence of vacuum states. However, as is well-known, the possible presence of vacuum is one of the major difficulties in the study of compressible MHD equations due to the high singularity and degeneracy of the system near the vacuum region. In this direction, it was shown in \cite{LH} that there exists a unique global classical solution to the 3D Cauchy problem with small initial energy but possibly large oscillations and vacuum states at both the interior domain and the far field. In addition, the small initial energy condition was relaxed in \cite{hou} and the global unique classical solution was proved under the smallness assumption of $\big[(\gamma-1)^{\frac19}+\nu^{-\frac14}\big]E_0$, where $E_0$ is the initial energy. Apart from the whole space case, the study of the compressible system in a domain with physical boundary conditions is also important mathematically and physically. More recently, in the 3D bounded domains and exterior domains with Navier-slip boundary conditions, the global existence and uniqueness of classical solutions to \eqref{a1} in the absence of external force with the regular initial data that are of small energy but possibly large oscillations and vacuum were studied in \cite{chen,chen1}.
On the other hand, for {\it large initial data}, Hu and Wang \cite{hu1} showed the global existence and large-time behavior of finite energy weak solutions with the adiabatic constant $\gamma>\f32$ and vacuum in 3D bounded domains under Dirichlet boundary conditions with the aid of weak convergence method developed by Lions \cite{PL} and Feireisl \cite{F04}. Yet the uniqueness and regularities for such weak solutions are major open questions.

The third type of solutions for compressible MHD equations are \textit{intermediate weak solutions} except \textit{small-classical solutions} \cite{LY11,ZZ10,chen,chen1,LH,hou} and \textit{large-energy weak solutions} \cite{hu1}. More precisely, by requiring initial data small in $L^2$ and initial density positive and essentially bounded, Suen and Hoff \cite{SH} proved the global existence of such solutions to \eqref{a1} without external force in $\mathbb{R}^3$ provided the viscosity coefficients $\mu$ and $\lambda$ satisfy an additional assumption
\begin{align}\label{a4}
\mu<2\mu+\lambda<\bigg(\frac32+\frac{\sqrt{21}}{6}\bigg)\mu.
\end{align}
This momentous result was later generalized in \cite{LS} where the condition \eqref{a4} was removed and vacuum was allowed initially.
It should be emphasized that the regularities of weak solutions obtained in \cite{SH} are stronger than finite-energy weak solutions in the sense that particle path can be defined in the non-vacuum region, but weaker than the usual strong solutions in the sense that discontinuity of the density can be transported along the particle path. Moreover, the uniqueness and continuous dependence of these weak solutions were established in \cite{S20} by modifying the initial conditions. Meanwhile, the optimal time-decay rates of such solutions were obtained in \cite{WZZ21}, that is, for $t\geq1$,
\begin{align*}
\begin{cases}
\|(\rho-\tilde{\rho},u,H)(t)\|_{L^r}\leq Ct^{-\frac32(1-\frac1r)},\ 2\leq r\leq\infty,\\
\|(\nabla u,\nabla H)(t)\|_{L^2}\leq Ct^{-\frac54}, \\
\|(\nabla^2H,H_t)(t)\|_{L^2}\leq Ct^{-\frac74}, \\
\end{cases}
\end{align*}
if $\|(\rho_0-\tilde{\rho},u_0,H_0)\|_{L^1}$ is bounded.

However, all results on global intermediate weak solutions \cite{LS,WZZ21,S20,SH} only concern the whole space case or the non-vacuum case. It remains completely open for the global existence of intermediate weak solutions to the isentropic compressible MHD flows with density containing vacuum initially in general bounded domains. This motivates us to investigate global weak solutions with discontinuous initial data and vacuum to the initial-boundary value problem \eqref{a1}--\eqref{a5}.

Before stating our main result precisely, we introduce the notations and conventions throughout.
The symbol denotes the end of a proof and $A\triangleq B$ means $A=B$ by definition. We denote by $\dot{g}\triangleq g_t+u\cdot\nabla g$ the material derivative of $g$ and write
\begin{align*}
\int fdx\triangleq\int_\Omega fdx, ~~\bar{f}\triangleq\frac{1}{|\Omega|}\int_\Omega fdx.
\end{align*}
For \(1 \leq p \leq \infty\) and \(0 \leq k \in \mathbb{Z}\), the standard Sobolev spaces are defined as follows
\[
\begin{cases}
L^p \triangleq L^p(\Omega), \ W^{k,p} \triangleq W^{k,p}(\Omega), \ H^k \triangleq W^{k,2}, \\
H_{\omega}^1\triangleq \{v\in H^1: (v \cdot n)|_{\partial\Omega}= (\curl v\times n)|_{\partial\Omega}= 0\}, \\
D^1 \triangleq \{v\in L^1_{\text{loc}}: \|\nabla v\|_{L^2}<\infty\},\ D^{1,p} \triangleq \{ v\in L^1_{\text{loc}}: \|\nabla v\|_{L^p}<\infty\}.
\end{cases}
\]
It is natural to expect an equilibrium state with density $\rho_s=\rho_s(x)$, velocity $u_s=u_s(x)$, and magnetic field $H_s=H_s(x)$ for the initial-boundary value problem \eqref{a1}--\eqref{a5}, which is a solution of
\begin{align}\label{1.8}
\begin{cases}
\divv(\rho_s u_s) = 0, & x\in\Omega, \\
\rho_s u_s\cdot\nabla u_s -\mu \Delta u_s - (\lambda + \mu) \nabla \divv u_s+ H_s\cdot\nabla H_s- \frac{1}{2}\nabla|H_s|^2 + \nabla P(\rho_s) = \rho_s \nabla \phi, & x\in\Omega, \\
-\nu\Delta H_s= H_s\cdot\nabla u_s- u_s\cdot\nabla H_s- H_s\divv u_s, & x\in\Omega,\\
\divv H_s= 0, & x\in\Omega,\\
u_s \cdot n = 0, \ \curl u_s \times n = 0, \
H_s \cdot n = 0, \ \curl H_s \times n = 0, & x\in\partial \Omega, \\
\int_{\Omega} \rho_s \, dx = \int_{\Omega} \rho_0 \, dx.
\end{cases}
\end{align}
Indeed, we have the following conclusion.
\begin{lemma}[\!\!{\cite{LY}}]\label{lem:a1}
Assume that $\Omega\subset\mathbb{R}^3$ is a bounded domain with smooth boundary and $\phi$ satisfies
\begin{equation}\label{1.9}
\phi \in H^2, \quad  \int \left(\frac{\gamma -1}{A \gamma} \Big(\phi - \inf\limits_{\Omega} \phi\Big) \right)^{\frac{1}{\gamma - 1}} \, dx < \int \rho_0 \, dx,
\end{equation}
then there exists a unique solution $(\rho_s, 0, 0)$ of \eqref{1.8} such that
\begin{equation}
\rho_s \in H^2, \quad 0 < \underline{\rho} \leq \inf\limits_{\Omega} \rho_s  \leq \sup\limits_{\Omega} \rho_s \leq \bar\rho,
\end{equation}
where $\underline{\rho}$ and $\bar\rho$ are positive constants depending only on $A$, $\gamma$, $\inf\limits_{\Omega} \phi$, and $\sup\limits_{\Omega} \phi$. In addition, if $\phi \in W^{2, q}$ for some $q \in (3, 6)$, then
\begin{equation}
\|\rho_s\|_{W^{2, q}} \leq C,
\end{equation}
where $C$ is a positive constant depending only on $A$, $\gamma$, $\inf\limits_{\Omega} \phi$, and $\|\phi\|_{W^{2, q}}$.
\end{lemma}

It should be noted that the equilibrium density $\rho_s$ is in fact a solution of the rest state equations
\begin{align}
\begin{cases}\label{state density}
\nabla P(\rho_s) = \rho_s \nabla \phi,\ x\in\Omega, \\
\int \rho_s \, dx = \int \rho_0 \, dx.
\end{cases}
\end{align}
We denote the initial total energy of \eqref{a1} as
\begin{equation}
C_0 \triangleq \int \bigg(\frac{1}{2} \rho_0 |u_0|^2 + G(\rho_0, \rho_s)+ \f12|H_0|^2\bigg) \, dx,
\end{equation}
where $G(\rho, \rho_s)$ is the potential energy density given by
\begin{equation}
G(\rho, \rho_s) \triangleq \rho \int_{\rho_s}^{\rho} \frac{P(\xi) - P(\rho_s)}{\xi^2} \, d\xi.
\end{equation}
It is easy to check that there exists a positive constant $C= C(\underline\rho, \bar\rho)$ such that
\begin{align}
C^{-1}(\rho -\rho_s)^2 \leq G(\rho) \leq C(\rho -\rho_s)^2.
\end{align}

We recall the weak solutions of \eqref{a1}--\eqref{a5} as defined in \cite{SH}.
\begin{defn}\label{def:weak-solution}
A triple $(\rho, u, H)$ is said to be a global weak solution of \eqref{a1}--\eqref{a5} provided that
$(\rho - \rho_s, \rho u, H) \in C([0, \infty); H^{-1}(\Omega))$,
$\nabla u \in L^2(\Omega\times(0,\infty))$,
$H \in L^\infty(0, \infty; L^2) \cap L^2(0, \infty; D^1)$, $\phi$ satisfies \eqref{1.9}, and  $\divv H(\cdot, t) = 0$ in $\mathcal{D}'(\Omega)$ for $t > 0$.
Moreover, the following identities hold for any test function $\psi \in \mathcal{D}(\Omega \times (t_1, t_2))$ with $t_2 \geq t_1 \geq 0$ and $j = 1, 2, 3$,
\begin{gather}
\left.\int \rho \psi(x, t) dx\right|_{t_1}^{t_2} = \int_{t_1}^{t_2} \int \left(\rho \psi_t + \rho u \cdot \nabla \psi\right) dx dt, \label{eq:weak1} \\
\left.\int \rho u^j \psi(x, t) dx\right|_{t_1}^{t_2} + \int_{t_1}^{t_2} \int \left(\mu \nabla u^j \cdot \nabla \psi + (\mu + \lambda)(\divv u) \psi_{x_j}\right)dx dt \notag\\
= \int_{t_1}^{t_2} \int \Big(\rho u^j \psi_t + \rho u^j u \cdot \nabla \psi + P(\rho) \psi_{x_j} + \frac{1}{2}|H|^2 \psi_{x_j} - H^j H \cdot \nabla \psi+\rho\phi_{x_j}\psi\Big) dx dt, \label{eq:weak2} \\
\left.\int H^j \psi(x, t) dx\right|_{t_1}^{t_2} = \int_{t_1}^{t_2} \int \left(H^j \psi_t + H^j u \cdot \nabla \psi - u^j H \cdot \nabla \psi - \nu \nabla H^j \cdot \nabla \psi\right) dx dt. \label{eq:weak3}
\end{gather}
\end{defn}

We now state our main result on the global existence of weak solutions.
\begin{theorem}\label{thm:main}
Let $\Omega$ be a simply connected bounded domain in $\mathbb{R}^3$ and its smooth boundary $\partial\Omega$ has a finite number of 2-dimensional connected components. Let $\phi$ satisfy \eqref{1.9} with $\rho_s$ being the steady state density providing \eqref{state density}.
For given positive numbers $M_1, M_2$ (not necessarily small) and $\hat{\rho} \geq \bar\rho + 1$, assume that the initial data $(\rho_0, u_0, H_0)$ satisfies
\begin{align}
\begin{cases}
0 \leq \inf \rho_0 \leq \sup \rho_0 \leq \hat{\rho},\
(u_0, H_0)\in H_{\omega}^1, \ \divv H_0 = 0,\\
 \|\nabla u_0\|_{L^2}^2 \leq M_1,\ \|\nabla H_0\|_{L^2}^2 \leq M_2.
\end{cases}
\end{align}
There exists a positive constant \(\varepsilon\) depending on $\mu$, $\lambda$, $\nu$, $\gamma$, A, $\hat{\rho}$, $\Omega$, $\|\phi\|_{H^2}$, $\inf\limits_{\Omega} \phi$, $M_1$, and $ M_2$ such that if
\begin{equation}
C_0 \leq \varepsilon,
\end{equation}
then there is a global weak solution \((\rho, u, H)\) to \eqref{a1}--\eqref{a5} in the sense of Definition \ref{def:weak-solution} verifying
\begin{equation}
0 \leq \rho(x, t) \leq 2\hat{\rho}, \quad \forall x \in \Omega, \ t \geq 0.
\end{equation}
Moreover, for any $r\in [1,\infty)$, $p\in [1,6]$, and $t>0$, there exist positive constant $C$ depending only on $\mu$, $\lambda$, $\nu$, $\gamma$, A, $\hat{\rho}$, $\Omega$, $\|\phi\|_{H^2}$, $\inf\limits_{\Omega} \phi$, $M_1$, $ M_2$, $r$, and $p$ such that
\begin{equation}
\lim_{t\rightarrow\infty}\int\big(|\rho-\rho_s|^r+\sqrt\rho|u|^4+|H|^p\big)dx=0 .\label{1.21}
\end{equation}
where $r\in(2,\infty)$ and $p\in(2,6].$
\end{theorem}

\begin{remark}
It should be noted that there is no restriction on the potential force $\phi$ except \eqref{1.9} in Theorem \ref{thm:main}. Moreover, both large oscillations of the solutions and vacuum states are allowed.
\end{remark}

\begin{remark}
Our Theorem \ref{thm:main} generalizes the Cauchy problem \cite{LS} to the case of bounded domains. However, this is a non-trivial generalization because the analysis involves handling a number of surface integrals arising from the boundary. In addition, the presence of large external force introduces further difficulties.
\end{remark}

\begin{remark}
The motivation for considering large external forces in our study is as follows. As stated by Li and Qin \cite{LT12}, for planetary-scale flows (such as in the atmosphere and ocean), the Coriolis force arising from rotation and the pressure gradient force constitute the dominant balance, while the inertial force acts only as a mere modification on this scale. Neglecting the dominance of these forces makes it impossible to capture the most essential characteristics of such flows.
\end{remark}

We now outline the core framework and main mathematical difficulties in proving Theorem \ref{thm:main}.
Our strategy is to extend the standard local classical solutions (with strictly positive initial density) globally in time provided the initial energy is suitably small, and then let the lower bound of the initial density go to zero.
To this end, the pivotal aspect of our analysis is to derive uniform {\it a priori} estimates, independent of the lower bound of the density $\rho$.
In contrast to the Cauchy problem in \cite{LS,LH}, our setting features both a physical boundary and a meaningful external force, thus requiring some new observations and approaches.

The first challenge arises from the non-trivial external force, which will impose major difficulties due to its powerful effects including magnetic fields on the dynamic motion of compressible flows.
On the one hand, with the aid of the effective viscous flux
\begin{align*}
F\triangleq (2\mu+\lambda)\divv u-(P-P(\rho_s))-\f12|H|^2,
\end{align*}
we can rewrite equation $\eqref{a1}_2$ into
\begin{align}\label{1.23}
\rho\dot u-(\rho- \rho_s)\nabla \phi= \nabla F- \mu\curl\curl u+ H\cdot\nabla H,
\end{align}
which combined with the boundary conditions \eqref{a5} allows us to obtain $L^p$-estimates for $\nabla F$. The elliptic theory also gives the bound of
$\|\nabla \curl u\|_{L^p}$, which exhibits higher regularity compared to the weak solutions obtained in \cite{LH}.
Furthermore, this together with the inequality
\begin{align*}
\|\nabla u\|_{W^{k, q}} \leq C(\|\divv u\|_{W^{k, q}}+ \|\curl u\|_{W^{k, q}}) \ \ \text{for any} \ q > 1, k \geq 0
\end{align*}
shown in \cite{ZLZ20} for simply connected domains enables the derivation of estimates for $\|\nabla u\|_{L^p}$ after the nonlinear terms in $\eqref{a1}_3$ like $\nabla\times(u\times H)$ and $\nabla\times(\nabla\times H)$ being properly addressed.
On the other hand, inspired by \cite{HLX}, we multiply $\eqref{a1}_2$ by $\rho_s^{-1}\mathcal{B}[\rho - \rho_s]$ (see Lemma \ref{lem:2.2}) to expand the pressure term $P$ into the Taylor series of order 2 (see \eqref{3.65}--\eqref{3.66}), which confers the global integrability
and certain smallness property on $\|\rho- \rho_s\|_{L^2(\Omega\times(0, T))}$ (see \eqref{3.64}).
This plays an important role in our analysis and, moreover, gives a description of the decay
rate of the density $\rho$.

The second difficulty stems from the presence of boundary, which gives rise to several boundary integral terms like $\int_{\partial\Omega}F_t(\dot u\cdot n)dS$ and $\int_{\partial\Omega} F \dot u\cdot\nabla u\cdot n dS$. Such boundary terms do not admit direct estimations, as the trace theorem forces the use of higher-order derivatives in the interior.
To overcome it, following \cite{CGC}, the condition $(u\cdot n)|_{\partial\Omega}= 0$ yields
\begin{equation*}
u\cdot\nabla u\cdot n= -u\cdot \nabla n\cdot u, \quad \text{on} \ \partial \Omega,
\end{equation*}
and further
\begin{equation*}
(\dot u+ (u\cdot\nabla n)\times u^{\bot})\cdot n= 0,  \quad \text{on} \ \partial \Omega
\end{equation*}
with $u^{\bot}= -u\times n$, which circumvents the appearance of partial first-order derivatives in the boundary integral (see, e.g., \eqref{3.37}--\eqref{3.40}).

From \eqref{1.23} and the elliptic equation of $F$ (see \eqref{ee1}),
our goal is the bounds of $\|\nabla u\|_{L^p}$ and $\|\nabla H\|_{L^p}$,
and time-weighted bounds of $\|\rho\dot u\|_{L^p}$ and $\|H_t\|_{L^p}$ motivated by previous work
\cite{CGC,LS,chen,hou}, which leads us to make the {\it a priori hypothesis} (see \eqref{3.1}--\eqref{3.3}). Having above estimates at hand, we get
the desired bound of $\|F\|_{L^\infty}$. Then, we can obtain the key upper bound of density by using the Zlotnik inequality (see Lemma \ref{lem:2.8}) and thus close the {\it a priori hypothesis}
in a standard manner. It should be emphasized that the effective viscous flux plays an essential role in our analysis.

The rest of the paper is arranged as follows. In the next section, we recall some known facts and elementary inequalities which will be used later. Section \ref{sec3} is devoted to obtaining the {\it a priori} estimates, while the proof of Theorem \ref{thm:main} is given in Section \ref{sec4}.

\section{Preliminaries}\label{sec2}
 In this section we recall some known facts and elementary inequalities that will be used later. We start with the following Gagliardo--Nirenberg inequality (see \cite[Lemma 2.3]{LWZ}).
\begin{lemma}\label{lem:2.1}
(Gagliardo--Nirenberg inequality, special case).
Assume that $\Omega$ is a bounded Lipschitz domain in $\mathbb{R}^3$. For $p\in [2, 6]$, $q\in (1, \infty)$, and $r\in (3, \infty)$, there exist generic constants $C_i>0\ (i\in\{1,2,3,4\})$ which may depend only on $p$, $q$, $r$, and $\Omega$ such that, for any $f\in H^1$ and $g\in L^q\cap D^{1, r}$,
\begin{gather}
\|f\|_{L^p}\le C_1\|f\|_{L^2}^\f{6-p}{2p}\|\na f\|_{L^2}^\f{3p-6}{2p}+C_2\|f\|_{L^2},\label{f1}\\
\|g\|_{L^\infty}\le C_3\|g\|_{L^q}^{\frac{q(r-3)}{3r+q(r-3)}}\|\na g\|_{L^r}^{\frac{3r}{3r+q(r-3)}}+C_4\|g\|_{L^2}.\label{f2}
\end{gather}
Moreover, if $\int_{\Omega}f(x)dx=0$ or $(f\cdot n)|_{\partial\Omega}=0$ or $(f\times n)|_{\partial\Omega}=0$, we can choose $C_2=0$. Similarly, the constant $C_4=0$ provided $\int_{\Omega}g(x)dx=0$ or $(g\cdot n)|_{\partial\Omega}=0$ or $(g\times n)|_{\partial\Omega}=0$.
\end{lemma}
Next, we introduce the Bogovskii operator in a bounded domain, which plays an important role in controlling $\|P\|_{L^2}^2$. One has the following conclusion (cf. \cite[Theorem III.3.1]{G}).
\begin{lemma}\label{lem:2.2}
Let $\Omega\subset\mathbb{R}^3$ be a bounded Lipschitz domain. Then, for any $p\in(1,\infty)$, there exists a linear operator $\mathcal{B}=\left[\mathcal{B}_{1}, \mathcal{B}_{2}, \mathcal{B}_{3}\right]: L^{p}(\Omega) \rightarrow \big(W_{0}^{1, p}(\Omega)\big)^{3}$ such that
\begin{align*}
\begin{cases}
\operatorname{div} \mathcal{B}[f]=f, & x \in \Omega, \\
\mathcal{B}[f]=0, & x \in \partial \Omega,
\end{cases}
\end{align*}
and
\begin{align*}
\|\nabla\mathcal{B}[f]\|_{L^{p}} \leq C(p,\Omega)\|f\|_{L^{p}}.
\end{align*}
Moreover, if $f=\operatorname{div} g$ with $g\in L^p(\Omega)$ satisfying $(g \cdot n)|_{\partial \Omega}=0$, it holds that
\begin{align*}
\|\mathcal{B}[f]\|_{L^{p}} \leq C(p,\Omega)\|g\|_{L^{p}}.
\end{align*}
\end{lemma}

For the Lam\'e system
\begin{align}\label{lame}
\begin{cases}
-\mu\Delta u - (\lambda + \mu)\nabla\divv u = f, \quad x \in \Omega,\\
u \cdot n = 0, \quad \curl u \times n = 0, \quad x \in \partial\Omega,
\end{cases}
\end{align}
one has the following estimates (cf. \cite{a}).
\begin{lemma}\label{l23}
Let \(u\) be a solution of the Lam\'e system \eqref{lame}, there exists a positive constant \(C\) depending only on \(\lambda\), \(\mu\), \(q\), \(k\), and \(\Omega\) such that
\begin{enumerate}
    \item If \(f \in W^{k,q}\) for some \(q \in (1, \infty)\) and integer \(k \geq 0\), then \(u \in W^{k+2,q}\) and
    \begin{equation}
    \|u\|_{W^{k+2,q}} \leq C(\|f\|_{W^{k,q}} + \|u\|_{L^q}).
    \end{equation}
    \item If \(f = \nabla g\) with \(g \in W^{k,q}\) for some \(q \geq 1\) and integer \(k \geq 0\), then \(u \in W^{k+1,q}\) and
    \begin{equation}
    \|u\|_{W^{k+1,q}} \leq C(\|g\|_{W^{k,q}} + \|u\|_{L^q}).
    \end{equation}
\end{enumerate}
\end{lemma}

The following two lemmas are given in \cite[Theorem 3.2]{ZLZ20} and \cite[Propositions 2.6--2.9]{A}.
\begin{lemma}\label{lem:2.4}
Let $k\ge 0$ be an integer and $1<q<+\infty$. Assume that $\Omega\subset\mathbb R^3$ is a simply connected bounded domain with $C^{k+1, 1}$ boundary $\pa\Omega$. Then, for $v\in W^{k+1, q}$ with $(v\cdot n)|_{\partial\Omega}=0$, it holds that
\begin{align}\label{ff5}
\|v\|_{W^{k+1, q}}\le C(\|\divv v\|_{W^{k, q}}+\|\curl v\|_{W^{k, q}}).
\end{align}
 In particular, for $k=0$, we have
 \begin{align}
 \|\na v\|_{L^q}\le C(\|\divv v\|_{L^q}+\|\curl v\|_{L^q}).
 \end{align}
\end{lemma}

\begin{lemma}\label{lm25}
Let $k\geq0$ be an integer and $1<q<+\infty$. Suppose that $\Omega\subset\mathbb R^3$ is a bounded domain and its $C^{k+1,1}$ boundary $\partial\Omega$ has a finite number of 2-dimensional connected components. Then, for $v\in W^{k+1,q}$ with $(v\times n)|_{\partial\Omega}=0$, we have
$$\|v\|_{W^{k+1,q}}\leq C(\|\divv v\|_{W^{k,q}}+\|\curl v\|_{W^{k,q}}+\|v\|_{L^q}).$$
In particular, if  $\Omega$ has no holes, then
$$\|v\|_{W^{k+1,q}}\leq C(\|\divv v\|_{W^{k,q}}+\|\curl v\|_{W^{k,q}}).$$
\end{lemma}

Motivated by \cite{SH}, we set
\begin{align}\label{f5}
F\triangleq (2\mu+\lambda)\divv u-(P-P(\rho_s))-\f12|H|^2 , \quad \omega\triangleq\na\times u,
\end{align}
where $F$ and $\omega$ denote the {\it effective viscous flux} and the {\it vorticity}, respectively. For $F$, $\omega$, and $\na u$, we have the following key {\it a priori} estimates.
\begin{lemma}\label{lem:2.6}
Let \((\rho, u, H)\) be a smooth solution of \(\boldsymbol{\eqref{a1}}\)--\(\boldsymbol{\eqref{a5}}\). Then for any \(p \in [2, 6]\), there exists a positive constant \(C\) depending only on \(p, q, \mu, \lambda, \nu\), the domain \(\Omega\), and \(\|\phi\|_{H^2}\) such that the following estimates hold
\begin{gather}
\|\nabla u\|_{L^p}\leq C(\|\divv u\|_{L^p}+\|\curl u\|_{L^p}),\label{111}\\
\|\nabla H\|_{L^p}\leq C\|\curl H\|_{L^p},\label{112}\\
\|\nabla F\|_{L^p}\leq C\big(\|\rho\dot u\|_{L^p}+\|\rho-\rho_{s}\|_{L^{\frac{6p}{6-p}}}+\|H\cdot \nabla H\|_{L^p}\big),\label{2.12}\\
\|\nabla\curl u\|_{L^p}\leq C\big(\|\rho\dot u\|_{L^p}+\|\rho-\rho_{s}\|_{L^{\frac{6p}{6-p}}}+\|H\cdot \nabla H\|_{L^p}+\|\nabla u\|_{L^2}\big),\label{114}\\
\|F\|_{L^p}\leq C(\|\rho\dot u\|_{L^2}+\|H\cdot \nabla H\|_{L^2})^{\frac{3p-6}{2p}}\big(\|\nabla u\|_{L^2}+\|\rho-\rho_{s}\|_{L^2}+\|H\|_{L^4}^{2}\big)^{\frac{6-p}{2p}}
\notag \\
+(\|\nabla u\|_{L^2}+\|\rho-\rho_{s}\|_{L^p}+\|\rho-\rho_{s}\|_{L^3}+\|H\|_{L^{4}}^{2}),\label{115}\\
\|\curl u\|_{L^p}\leq C\big(\|\rho\dot u\|_{L^{2}}+\|H\cdot\nabla H\|_{L^{2}}\big)^{\frac{3p-6}{2p}}\|\nabla u\|_{L^{2}}^{\frac{6-p}{2p}}+C\big(\|\nabla u\|_{L^2}+\|\rho-\rho_{s}\|_{L^3}\big),\\
\|\nabla u\|_{L^p}\leq C\big(\|\rho\dot u\|_{L^2}+\|H\cdot\nabla H\|_{L^2}\big)^{\frac{3p-6}{2p}}\big(\|\nabla u\|_{L^2}+\|\rho-\rho_s\|_{L^2}+\|H\|_{L^4}^2\big)^{\frac{6-p}{2p}} \label{116} \notag \\
+C\big(\|\nabla u\|_{L^2}+\|\rho-\rho_s\|_{L^p}+\|\rho-\rho_s\|_{L^3}+\|H\|_{L^{2p}}^2\big).
\end{gather}
\end{lemma}

\begin{remark}\label{remark:2.1}
If \(p = 6\), then \(\|\rho - \rho_s\|_{L^{\frac{6p}{6-p}}} \triangleq \|\rho - \rho_s\|_{L^\infty}\).
\end{remark}
\begin{proof}
We get \eqref{111} and \eqref{112} from \eqref{a5} and Lemma \ref{lem:2.4}.

By \eqref{a1}$_2$ and $\eqref{a5}$, one finds that the effective viscous flux $F$ satisfies
\begin{align}\label{ee1}
\begin{cases}
\Delta F=\div(\rho\dot u-H\cdot\nabla H-(\rho-\rho_s)\nabla\phi) &\text{in}~\Omega,\\
\f{\pa F}{\pa n}=(\rho\dot u-H\cdot\nabla H-(\rho-\rho_s)\nabla\phi)\cdot n &\text{on}~\pa\Omega.
\end{cases}
\end{align}
It follows from \cite[Lemma 4.27]{NS1} that
\begin{align}\label{ff14}
\|\na F\|_{L^p}&\le C\|\rho\dot u-H\cdot\nabla H-(\rho-\rho_s)\nabla\phi\|_{L^p}\notag\\
&\leq C\|\rho\dot u\|_{L^p}+C\|H\cdot\nabla H\|_{L^p}+C\|\rho-\rho_s\|_{L^{\frac{6p}{6-p}}}\|\nabla\phi\|_{L^6}\notag\\
&\leq C\big(\|\rho\dot u\|_{L^p}+\|H\cdot\nabla H\|_{L^p}+\|\rho-\rho_s\|_{L^{\frac{6p}{6-p}}}\big).
\end{align}
On the other hand, we rewrite $\eqref{a1}_2$ as
\begin{align*}
\mu\na\times \omega=\na F-\rho\dot u+H\cdot\nabla H+(\rho-\rho_s)\nabla\phi,
\end{align*}
which along with $(\omega\times n)|_{\pa\Omega}=0$, $\divv \omega=0$, and Lemma \ref{lm25} leads to
\begin{align}\label{ff16}
\|\nabla\curl u\|_{L^p}&\le C\big(\|\na\times\curl u\|_{L^p}+\|\curl u\|_{L^p}\big)\notag\\
&\leq C\big(\|\rho\dot u\|_{L^p}+\|H\cdot\nabla H\|_{L^p}+\|\rho-\rho_s\|_{L^{\frac{6p}{6-p}}}+\|\curl u\|_{L^p}\big)\notag\\
&\leq C\big(\|\rho\dot u\|_{L^p}+\|H\cdot\nabla H\|_{L^p}+\|\rho-\rho_s\|_{L^{\frac{6p}{6-p}}}+\|\nabla u\|_{L^2}\big),
\end{align}
where we have used
\begin{align}\label{z2.20}
\|\curl u\|_{L^p}&\le C\|\curl u\|_{L^2}^{\frac{6-p}{2p}}\|\nabla\curl u\|_{L^2}^{\frac{3p-6}{2p}}
+C\|\curl u\|_{L^2}\nonumber\\
&\leq C\big(\|\curl u\|_{L^2}+\|\nabla\curl u\|_{L^2}\big)\notag\\
&\leq C\big(\|\rho\dot u\|_{L^2}+\|H\cdot\nabla H\|_{L^2}+\|\rho-\rho_s\|_{L^{3}}+\|\nabla u\|_{L^2}\big).
\end{align}
\par Furthermore, one can deduce from \eqref{f1} and \eqref{2.12} that, for \( p \in [2,6] \),
\begin{align}
\| F \|_{L^p} &\leq C \| F \|_{L^2}^{\frac{6-p}{2p}} \| \nabla F \|_{L^2}^{\frac{3p-6}{2p}}+C\| F \|_{L^2}  \notag \\
&\leq C \left( \| \rho\dot u \|_{L^2} + \| \rho - \rho_s \|_{L^3}+\|H\cdot\nabla H\|_{L^2} \right)^{\frac{3p-6}{2p}} \left( \| \nabla u \|_{L^2} + \| \rho - \rho_s \|_{L^2} +\|H\|_{L^4}^2\right)^{\frac{6-p}{2p}} \notag \\
&\quad +C\left(\|\nabla u\|_{L^2}+ \|\rho- \rho_s\|_{L^3}+ \|H\|_{L^4}^2\right)\notag\\
&\leq C \left( \| \rho\dot u \|_{L^2}+ \|H\cdot\nabla H\|_{L^2} \right)^{\frac{3p-6}{2p}} \left( \| \nabla u \|_{L^2} + \| \rho - \rho_s \|_{L^2} +\|H\|_{L^4}^2\right)^{\frac{6-p}{2p}} \notag\\
&\quad +C\left(\|\nabla u\|_{L^2}+ \|\rho- \rho_s\|_{L^3}+ \|H\|_{L^4}^2\right),
\end{align}
which implies that
\begin{equation}
\| F \|_{L^p} \leq C \left( \| \rho\dot u \|_{L^2} + \| \nabla u \|_{L^2} + \| \rho - \rho_s \|_{L^3} +\|H\cdot\nabla H\|_{L^2}+\|H\|_{L^4}^2\right). \label{eq:2.18_simplified}
\end{equation}
Similarly, one gets that
\begin{align*}
\| \text{curl} \, u \|_{L^p} &\leq C \| \text{curl} \, u \|_{L^2}^{\frac{6-p}{2p}} \| \nabla \text{curl} \, u \|_{L^2}^{\frac{3p-6}{2p}} + C \| \text{curl} \, u \|_{L^2} \notag \\
&\leq C \left( \| \rho\dot u \|_{L^2} + \| \rho - \rho_s \|_{L^3} +\|H\cdot\nabla H\|_{L^2}+\| \nabla u \|_{L^2} \right)^{\frac{3p-6}{2p}} \| \nabla u \|_{L^2}^{\frac{6-p}{2p}}+ C \| \nabla u \|_{L^2} \notag \\
&\leq C \left( \| \rho\dot u \|_{L^2} +\|H\cdot\nabla H\|_{L^2} \right)^{\frac{3p-6}{2p}} \| \nabla u \|_{L^2}^{\frac{6-p}{2p}}+ C(\| \nabla u \|_{L^2}+\|\rho-\rho_s\|_{L^3}),
\end{align*}
which leads to
\begin{equation}
\| \text{curl} \, u \|_{L^p} \leq C \left( \| \rho\dot u \|_{L^2} + \| \nabla u \|_{L^2} + \| \rho - \rho_s \|_{L^3} +\|H\cdot\nabla H\|_{L^2}\right). \label{eq:2.19_simplified}
\end{equation}
By virtue of \eqref{ff5}, \eqref{114}, and \eqref{115}, it holds that
\begin{align}
\|\nabla u\|_{L^p} &\leq C(\|\divv u\|_{L^p}+\|\curl u\|_{L^p}) \notag \\
&\leq C\big( \|F\|_{L^p} + \|\curl u\|_{L^p} + \| P - P(\rho_s) \|_{L^p}+\||H|^2\|_{L^p}\big) \notag \\
&\leq C (\| \rho\dot u \|_{L^2}+\|H\cdot\nabla H\|_{L^2})^{\frac{3p-6}{2p}}\big(\| \nabla u \|_{L^2}+\|\rho-\rho_s\|_{L^2}+\|H\|_{L^4}^2\big)^{\frac{6-p}{2p}}\notag \\
&\quad + C\big(\| \nabla u \|_{L^2} + \| \rho - \rho_s \|_{L^3} + \| \rho - \rho_s \|_{L^p}+\|H\|_{L^{2p}}^2\big).
 \end{align}
This completes the proof.
\end{proof}

Furthermore, the following estimates for material derivative of $u$ have been obtained in \cite[Lemma 3.2]{CGC}.

\begin{lemma}\label{lem:2.7}
Under the assumption of Lemma \ref{l23}, there exists a positive constant $C$ depending only on $\Omega$ such that
\begin{gather}
\|\dot u\|_{L^6}\le C\big(\|\nabla\dot u\|_{L^2}+\|\nabla u\|_{L^2}^2\big),\label{2.26}\\
\|\nabla\dot u\|_{L^2}\le C\big(\|\divv\dot u\|_{L^2}+\|\curl\dot u\|_{L^2}+\|\nabla u\|_{L^4}^2\big).\label{d12}
\end{gather}
\end{lemma}

Finally, we need the following Zlotnik's inequality (see \cite[Lemma 1.3]{Z}), by which we can get the uniform (in time) upper bound of the density $\rho$.

\begin{lemma}\label{lem:2.8}
Suppose the function \( y \) satisfies
\begin{equation}
y'(t) = g(y) + b'(t) \quad \text{on } [0, T], \quad y(0) = y_{0},
\end{equation}
with \( g \in C(\mathbb{R}) \) and \( y, b \in W^{1,1}(0, T) \). If \( g(\infty) = -\infty \) and
\begin{equation}
b(t_{2}) - b(t_{1}) \leq N_{0} + N_{1}(t_{2} - t_{1})
\end{equation}
for all \( 0 \leq t_{1} < t_{2} \leq T \) with some \( N_{0} \geq 0 \) and \( N_{1} \geq 0 \), then
\begin{equation}
y(t) \leq \max \{y_{0}, \xi^*\} + N_{0} < \infty \quad \text{on } [0, T],\label{2.31}
\end{equation}
where \( \xi^* \) is a constant such that
\begin{equation}
g(\xi) \leq -N_{1} \quad \text{for } \xi \geq \xi^*.
\end{equation}
\end{lemma}

\section{ A prior estimates}\label{sec3}

In this section we will establish the time-independent {\it a priori} bounds of the solutions to the problem \eqref{a1}--\eqref{a5}. Let \( T > 0 \) be fixed and \( (\rho, u, H) \) be a smooth solution of \eqref{a1}--\eqref{a5} in \( \Omega\times (0, T] \). Moreover, we set \( \sigma(t) \triangleq \min\{1, t\} \) and define
\begin{align}
A_1(T) &\triangleq \sup_{0\leq t \leq T}\big[\sigma\big(\|\nabla u\|_{L^2}^2+\|\nabla H\|_{L^2}^2\big)\big]
+\int_0^T\sigma\big(\|\sqrt{\rho}\dot{u}\|_{L^2}^2+\|H_t\|_{L^2}^2+\|\curl^2 H\|_{L^2}^2 \big) dt, \label{3.1} \\
A_2(T) &\triangleq \sup_{0 \leq t \leq T}\big[\sigma^2\big(\| \sqrt{\rho}\dot{u}\|_{L^2}^2+\|\curl^2 H\|_{L^2}^2+\|H_t\|_{L^2}^2\big)\big]
+\int_0^T\sigma^2\big(\|\nabla\dot{u}\|_{L^2}^2+\|\nabla H_t\|_{L^2}^2\big)dt, \label{3.2}\\
A_3(T) &\triangleq \sup_{0 \leq t \leq T} \big( \|\nabla u\|_{L^2}^2 + \|\nabla H\|_{L^2}^2 \big).\label{3.3}
\end{align}

This section is entirely devoted to proving the following conclusion.

\begin{proposition}\label{prop:3.1}
Under the conditions of Theorem \ref{thm:main}, there exist positive constants \(\varepsilon\) and \(K\) depending on $\mu$, $\lambda$, $\gamma$, $\nu$, A, $\inf\limits_{\Omega}\phi$, $\| \phi \|_{H^2}$, $\hat{\rho}$, $\Omega$, $M_1$, and $M_2$ such that if \((\rho, u, H)\) is a smooth solution of \eqref{a1}--\eqref{a5} in \(\Omega \times (0, T]\) satisfying
\begin{equation}\label{3.4}
\sup_{\Omega \times [0, T]} \rho \leq 2\hat{\rho}, \quad A_1(T) + A_2(T) \leq 2C_0^{\f12}, \quad A_3(\sigma(T)) \leq 3K,
\end{equation}
then we have
\begin{equation}
\sup_{\Omega \times [0, T]} \rho \leq \frac{7\hat{\rho}}{4}, \quad A_1(T) + A_2(T) \leq C_0^{\f12}, \quad A_3(\sigma(T)) \leq 2K,
\end{equation}
provided \(C_0 \leq \varepsilon\).
\end{proposition}

\begin{proof}
Proposition \ref{prop:3.1} is a consequence of the following Lemmas \ref{lem:3.1}--\ref{lem:3.8}.
\end{proof}

We begin with the basic energy estimates of \( (\rho, u, H) \).
\begin{lemma}\label{lem:3.1}
Let $(\rho, u, H)$ be a smooth solution of \eqref{a1}--\eqref{a5} in $\Omega$ $\times$ $(0,T]$ satisfying $\rho\leq 2\hat\rho $, then there is a positive constant $C$ depending only on $\mu$, $\lambda$, $\gamma$, $\nu$, $\hat\rho$, and $\Omega$ such that
\begin{align}\label{3.7}
&\sup_{0 \leq t \leq T}\int\bigg(\f12\rho|u|^2+(\rho-\rho_s)^2+\f12|H|^2\bigg) dx \notag \\
&\quad+ \int_0^T \big[(\lambda+2\mu)\|\divv u\|_{L^2}^2 + \mu\|\curl u\|_{L^2}^2 + \nu\|\curl H\|_{L^2}^2\big]dt \leq C C_0.
\end{align}
\end{lemma}
\begin{proof}
Multiplying $\eqref{a1}_1$ by $G'(\rho)$ and integrating the resulting equality over $\Omega$, we have
\begin{align}\label{w8}
\frac{d}{dt} \int G(\rho)dx+\int P\divv udx
-\frac{A\gamma}{\gamma-1}\int\rho^{\gamma-1}_s\divv(\rho u)dx = 0.
\end{align}
According to \eqref{state density} and
\begin{align*}
\Delta u=\nabla\divv u-\curl\curl u,
\end{align*}
we can rewrite $\eqref{a1}_2$ as
\begin{gather}\label{lo}
\rho\dot u - (\lambda + 2\mu) \nabla \divv u + \mu \nabla \times \curl u + \nabla P =H\cdot\nabla H- \f12\nabla|H|^2+ \frac{A\gamma}{\gamma - 1} \rho \nabla \rho^{\gamma - 1}_s.
\end{gather}
Multiplying \eqref{lo} by \( u \), integrating the resulting equality over \( \Omega \), and using \eqref{a5}, one gets that
\begin{align}\label{l1}
& \f12\frac{d}{dt}\int\rho|u|^2dx+(\lambda+2\mu)\int(\divv u)^2 dx
+\mu\int|\curl u|^2dx \notag\\
&=\int P\divv udx-\int H\cdot\nabla u\cdot H dx+\f12\int|H|^2\divv udx
-\frac{A\gamma}{\gamma-1}\int\rho^{\gamma-1}_s\divv(\rho u)dx.
\end{align}
Multiplying $\eqref{a1}_3$ by $H$ and integrating the resultant over $\Omega$, we have
\begin{align}\label{cz1}
\f12\frac{d}{dt}\int |H|^2 dx +\nu\int|\curl H|^2 dx= \int H\cdot\nabla u\cdot H dx- \f12\int |H|^2\divv u dx.
\end{align}
Combining \eqref{w8}, \eqref{l1}, and \eqref{cz1} together and integrating the resultant over $(0,T)$, one then obtains the desired \eqref{3.7}.
\end{proof}

\begin{lemma}\label{lem:2}
Let the condition \eqref{3.4} be satisfied, then there exist positive constants $K$ and $\varepsilon_1$ depending on $\mu$, $\lambda$, $\nu$, $\gamma$, $A$, $\hat\rho$, $\|\phi\|_{H^2}$, and $\inf\limits_{\Omega} \phi$ such that
\begin{align}\label{3.10}
A_3(\sigma(T)) + \int_0^{\sigma(T)}\big(\|\sqrt{\rho}\dot{u}\|_{L^2}^2 + \|H_t\|_{L^2}^2 + \|\curl^2 H\|_{L^2}^2 \big) dt \leq 2K,
\\ \label{3.11}
\sup_{0 \leq t \leq T} \|\nabla H\|_{L^2}^2 + \int_0^T \big( \|H_t\|_{L^2}^2 + \|\curl^2 H\|_{L^2}^2 \big)dt \leq C \|\nabla H_0\|_{L^2}^2,
\end{align}
provided $C_0 \leq \varepsilon_1$.
\end{lemma}
\begin{proof}
Multiplying $\eqref{a1}_2$ by $u_t$ and integration by parts, we obtain that
\begin{align}\label{141}
& \frac{1}{2}\frac{d}{dt}\left((2\mu+\lambda)\|\divv u\|_{L^2}^2+\mu\|\curl u\|_{L^2}^2\right) +\|\sqrt{\rho}\dot{u}\|_{L^2}^2 \notag\\
&= \frac{d}{dt} \int \bigg( \frac{1}{2} |H|^2 \divv u - H \cdot \nabla u \cdot H + \left( P(\rho) - P(\rho_s) \right)\divv u+(\rho-\rho_s)\nabla\phi\cdot u\bigg) dx \notag\\
&\quad+ \int \big( H_t \cdot \nabla u \cdot H + H \cdot \nabla u \cdot H_t - H \cdot H_t \divv u\big) dx +\int\big(\rho u \cdot \nabla u \cdot \dot{u}+\rho_t\nabla\phi\cdot u
-P_t \divv u\big) dx.
\end{align}
It follows from $\eqref{a1}_3$ that
\begin{align}\label{3.13}
\nu \frac{d}{dt} \|\curl H\|_{L^2}^2+\nu^2\|\curl^2 H\|_{L^2}^2+ \|H_t\|_{L^2}^2
&= \int|H_t-\nu\Delta H|^2 dx \notag\\
&= \int|H\cdot\nabla u-u\cdot\nabla H-H\divv u|^2 dx,
\end{align}
which combined with \eqref{141} gives that
\begin{align}\label{3.12}
 &\frac{d}{dt}\bigg( \frac{1}{2} (2\mu+ \lambda) \|\divv u\|_{L^2}^2 +\f12\mu \|\curl u\|_{L^2}^2 + \nu \|\curl H\|_{L^2}^2 \bigg) + \|\sqrt{\rho} \dot{u}\|_{L^2}^2 + \|H_t\|_{L^2}^2 + \nu^2 \|\curl^2 H\|_{L^2}^2\notag \\
&= \frac{d}{dt} \int \bigg( \frac{1}{2} |H|^2 (\divv u) - H \cdot \nabla u \cdot H + \left( P(\rho) - P(\rho_s) \right)\divv u+ (\rho-\rho_s)\nabla\phi\cdot u \bigg) dx \notag\\
&\quad+ \int \left( H_t \cdot \nabla u \cdot H + H \cdot \nabla u \cdot H_t- H \cdot H_t\divv u \right) dx - \int P_t (\divv u) dx+ \int \rho u \cdot \nabla u \cdot \dot{u} dx \notag\notag\\
&\quad+ \int |H \cdot \nabla u - u \cdot \nabla H - H \divv u|^2 dx + \int \rho_t\nabla\phi\cdot u dx\triangleq \frac{d}{dt} I_0 + \sum_{i=1}^5 I_i.
\end{align}

Next, we estimate each term $I_i\ (i\in\{1,2,\cdots,5\})$. By Lemma \ref{lem:2.1} and Young's inequality, we have that, for any \( \delta > 0 \),
\begin{align}\label{I_1}
I_1 &\leq C \|H\|_{L^\infty} \|H_t\|_{L^2} \|\nabla u\|_{L^2} \notag\\
&\leq C \|\nabla H\|_{L^2}^{\f12} \|\curl^2 H\|_{L^2}^{\f12} \|H_t\|_{L^2} \|\nabla u\|_{L^2} +C\|\nabla u\|_{L^2}\|H_t\|_{L^2}\|\nabla H\|_{L^2}\notag\\
&\leq\delta\left(\|\curl^2 H\|_{L^2}^2+\|H_t\|_{L^2}^2 \right)+ C(\delta)\big(\|\nabla u\|_{L^2}^2+ 1\big)\|\nabla u\|_{L^2}^2\|\nabla H\|_{L^2}^2.
\end{align}
By $\eqref{a1}_1$ and $P=A\rho^\gamma$, we arrive at
\begin{equation}\label{3.16}
P_t = -u \cdot \nabla P -\gamma P \divv u,
\end{equation}
which together with Lemma \ref{lem:2.6}, \eqref{f5}, \eqref{3.4}, and \eqref{3.7} yields that
\begin{align}\label{I_2}
I_2 &=\int\bigg[\gamma P (\divv u)^2 + \frac{1}{2\mu + \lambda} \Big( F + P - P(\tilde{\rho}) + \frac{1}{2} |H|^2 \Big) \left( u \cdot \nabla P  \right) \bigg]dx \notag\\
&\leq C\|\nabla u\|_{L^2}^2- \frac{1}{2\mu+\lambda}\int \big(\nabla F+ \nabla (P- P(\rho_s))+ H^j\nabla H^j \big)\cdot(P-P(\rho_s))u dx \notag\\
&\quad - \frac{1}{2\mu+\lambda}\int \Big(F+(P-P(\rho_s))+ \f12|H|^2\Big )\divv u(P-P(\rho_s)) dx\notag\\
&\quad + \frac{1}{2\mu+\lambda}\int \Big(F+(P-P(\rho_s))+\f12|H|^2\Big )\rho_s\nabla\phi\cdot u dx\notag\\
&\leq C\big(\|\nabla u\|_{L^2}^2+ \|u\|_{L^6}\|\nabla F\|_{L^2}\|P-P(\rho_s)\|_{L^3}+ \|u\|_{L^6}\||H||\nabla H|\|_{L^2}\|P-P(\rho_s)\|_{L^3}\notag\\
&\quad+ \|\nabla u\|_{L^2}\|P-P(\rho_s)\|_{L^4}^2+ \|\nabla u\|_{L^2}\|F\|_{L^2}\|P-P(\rho_s)\|_{L^\infty}+ \|H\|_{L^4}^2\|P-P(\rho_s)\|_{L^\infty}\|\nabla u\|_{L^2}\notag\\
&\quad+ \|F\|_{L^2}\|u\|_{L^6}\|\nabla\phi\|_{L^3}+ \|P-P(\rho_s)\|_{L^3}\|u\|_{L^6}\|\nabla\phi\|_{L^2}+ \|H\|_{L^6}^2\|u\|_{L^6}\|\nabla\phi\|_{L^2}\big)\notag\\
&\leq C\big(\|\nabla u\|_{L^2}^2+\|P-P(\rho_s)\|_{L^4}^4\big)+ C\|\nabla u\|_{L^2}\big(\|\nabla F\|_{L^2}+ \||H||\nabla H|\|_{L^2}\big)\|P-P(\rho_s)\|_{L^3}\notag\\
&\quad+ C\|\nabla u\|_{L^2}\big(\|\nabla u\|_{L^2}+\|\rho-\rho_s\|_{L^3}+ \|\nabla H\|_{L^2}^2\big)\notag\\
&\leq C\big(\|\nabla u\|_{L^2}^2+ C_0\big)+ CC_0^{\f13}\|\nabla u\|_{L^2}\big(\|\rho\dot u\|_{L^2}+ \|\rho-\rho_s\|_{L^3}+ \||H||\nabla H|\|_{L^2}\big)\notag\\
&\quad+ C\|\nabla u\|_{L^2}\big(\|\nabla u\|_{L^2}+\|\rho-\rho_s\|_{L^3}+ \|\nabla H\|_{L^2}^2\big)\notag\\
&\leq \delta\big(\|\rho\dot u\|_{L^2}^2+ \|\curl^2 H\|_{L^2}^2\big)+ C(\delta)\Big(C_0^\f23+ \|\nabla u\|_{L^2}^2+ \|\nabla u\|_{L^2}^{\f43}\|\nabla H\|_{L^2}^2+ \|\nabla u\|_{L^2}\|\nabla H\|_{L^2}^2\Big),
\end{align}
where we have used
\begin{align}\label{3.19}
\||H||\nabla H|\|_{L^2} \leq C \|H\|_{L^\infty} \|\nabla H\|_{L^2}
\leq C\Big(\|\nabla H\|_{L^2}^{\f32} \|\curl^2 H\|_{L^2}^{\f12}+\|\nabla H\|_{L^2}^2\Big).
\end{align}
Using \eqref{f1}, Lemma \ref{lem:2.6}, \eqref{3.4}, \eqref{3.7}, and \eqref{3.19}, we deduce that
\begin{align}\label{I_3}
I_3&= \int \rho u\cdot\nabla u\cdot\dot u dx\notag\\
&\leq C\|\sqrt{\rho}\dot{u}\|_{L^2} \|\nabla u\|_{L^3} \|u\|_{L^6}
\leq C\|\sqrt{\rho}\dot{u}\|_{L^2} \|\nabla u\|_{L^2}^{\f32} \|\nabla u\|_{L^6}^{\f12} \notag\\
&\leq C\|\sqrt{\rho}\dot u\|_{L^2}\|\nabla u\|_{L^2}^{\f32}\Big(\|\sqrt{\rho}\dot u\|_{L^2}^{\f12}+ \||H||\nabla H|\|_{L^2}^{\f12}+\|\rho-\rho_s\|_{L^6}^{\f12}+\|\nabla u\|_{L^2}^{\f12}\Big)\notag\\
&\leq C\|\sqrt{\rho}\dot{u}\|_{L^2} \|\nabla u\|_{L^2}^{\f32} \Big( 1 + \|\sqrt{\rho}\dot{u}\|_{L^2}^{\f12} + \|\nabla H\|_{L^2}^{\f34} \|\curl^2 H\|_{L^2}^{\f14}+ \|\nabla H\|_{L^2}+ \|\nabla u\|_{L^2}^{\f12} \Big)\notag \\
& \leq \delta \big( \|\sqrt{\rho}\dot{u}\|_{L^2}^2+\|\curl^2 H\|_{L^2}^2 \big)+ C(\delta)\big(1 + \|\nabla u\|_{L^2}^6  + \|\nabla u\|_{L^2}^4 \|\nabla H\|_{L^2}^2 + \|\nabla H\|_{L^2}^6\big),
\\ \label{I_4}
I_4 &\leq C \|u\|_{L^6}^2 \|\nabla H\|_{L^3}^2 + C\|\nabla u\|_{L^2}^2 \|H\|_{L^\infty}^2 \notag\\
&\leq C\big(\|\nabla u\|_{L^2}^2 \|\nabla H\|_{L^2} \|\curl^2 H\|_{L^2}+ \|\nabla u\|_{L^2}^2\|\nabla H\|_{L^2}^2\big)\notag \\
&\leq \delta \|\curl^2 H\|_{L^2}^2 + C(\delta)\big(\|\nabla u\|_{L^2}^2+1\big)\|\nabla u\|_{L^2}^2 \|\nabla H\|_{L^2}^2,
\\ \label{I_5}
I_5 &\leq \int\divv (\rho u)\nabla\phi\cdot u dx = -\int \rho u\cdot\nabla (\nabla\phi\cdot u)dx\notag\\
&\leq C\|u\|_{L^4}^2\|\nabla^2\phi\|_{L^2}+ C\|u\|_{L^3}\|\nabla u\|_{L^2}\|\nabla\phi\|_{L^6}\leq C\|\nabla u\|_{L^2}^2.
\end{align}

Hence, putting \eqref{I_1}, \eqref{I_2}, and \eqref{I_3}--\eqref{I_5} into \eqref{3.12}, integrating the resultant over $(0, \sigma(T))$, and choosing $\delta > 0$ sufficiently small, we conclude from Lemma \ref{lem:2.4}, \eqref{3.4}, \eqref{3.7}, and Young's inequality that
\begin{align}\label{3.24}
&A_3(\sigma(T)) + \int_0^{\sigma(T)}\big(\|\sqrt{\rho}{u}\|_{L^2}^2 + \|H_t\|_{L^2}^2 + \|\curl^2 H\|_{L^2}^2 \big)dt\notag\\
&\leq C \big(1+\|\nabla u_0\|_{L^2}^2 + \|\nabla H_0\|_{L^2}^2 \big)+C C_0^{\f12}\sup_{0 \leq t \leq \sigma(T)} \|\nabla H\|_{L^2}^3 \notag\\
&\quad + C\int_0^{\sigma(T)}\big(\|\nabla u\|_{L^2}^6+\|\nabla u\||_{L^2}^3+ \|\nabla u\|_{L^2}^2+ \|\nabla u\|_{L^2}^4\|\nabla H\|_{L^2}^2+\|\nabla u\|_{L^2}^{\f43}\|\nabla H\|_{L^2}^2 \notag\\
&\quad + \|\nabla u\|_{L^2}\|\nabla H\|_{L^2}^2 + \|\nabla H\|_{L^2}^2+ \|\nabla u\|_{L^2}^2\|\nabla H\|_{L^2}^2+ \|\nabla H\|_{L^2}^3+ \|\nabla H\|_{L^2}^6\big)dt \notag\\
&\leq C(M_1, M_2) + C C_0^{\f12} A_3^{\f32}(\sigma(T))+ CC_0A_3^{2}(\sigma(T))  \notag\\
&\leq K + C C_0^{\frac12} A_3^2(\sigma(T))
\end{align}
with \( K \triangleq C(M_1, M_2) \), where we have used
\begin{align}\label{I_0}
I_0 &\leq C \|\nabla u\|_{L^2} \left( \|P - P(\tilde{\rho})\|_{L^2} + \|H\|_{L^4}^2 \right)+ C\|u\|_{L^6}\|\rho-\rho_s\|_{L^3}\|\nabla\phi\|_{L^2} \notag\\
&\leq C \|\nabla u\|_{L^2}\Big( C_0^{\f12} + \|H\|_{L^2}^{\f12} \|\nabla H\|_{L^2}^{\f32}\Big)+ CC_0^{\f13}\|\nabla u\|_{L^2} \notag\notag\\
&\leq \delta\|\nabla u\|_{L^2}^2+ C(\delta)\Big(C_0^{\f23}+C_0^{\f12}\|\nabla H\|_{L^2}^3\Big)
\end{align}
due to \eqref{3.4}, \eqref{3.7}, and \eqref{f1}.
As an immediate consequence of \eqref{3.24}, one obtains \eqref{3.10} provided that\\
\begin{equation*}
A_3(\sigma(T)) \leq 3K \quad \text{and} \quad C_0 \leq \varepsilon_{1,1} \triangleq \min\{1, \, (9C K)^{-2}\}.
\end{equation*}

Finally, we shall prove \eqref{3.11}. One deduces from \eqref{3.13} and \eqref{I_4} that
\begin{align}\label{3.28}
\sup_{0 \leq t \leq T} \|\nabla H\|_{L^2}^2 + \int_0^T \left( \|H_t\|_{L^2}^2 + \|\curl^2 H\|_{L^2}^2 \right) dt &\leq C \|\nabla H_0\|_{L^2}^2 + C_1 \sup_{0 \leq t \leq T} \|\nabla H\|_{L^2}^2 \int_0^T \|\nabla u\|_{L^2}^4 dt.
\end{align}
Moreover, it follows from \eqref{3.4}, \eqref{3.7}, and \eqref{3.10} that
\begin{align}\label{3.29}
&\int_0^T \big(\|\nabla u\|_{L^2}^4+ \|\nabla H\|_{L^2}^4 \big)dt = \int_0^{\sigma(T)} \big(\|\nabla u\|_{L^2}^4+ \|\nabla H\|_{L^2}^4 \big)dt + \int_{\sigma(T)}^T \big(\|\nabla u\|_{L^2}^4+ \|\nabla H\|_{L^2}^4\big) dt \notag\\
&\leq \sup_{0 \leq t \leq \sigma(T)} \|\nabla u\|_{L^2}^2 \int_0^{\sigma(T)} \|\nabla u\|_{L^2}^2 dt + \sup_{\sigma(T) \leq t \leq T}\big( \sigma \|\nabla u\|_{L^2}^2 \big)\int_{\sigma(T)}^T \|\nabla u\|_{L^2}^2 dt \notag\\
&\quad + \sup_{0 \leq t \leq \sigma(T)} \|\nabla H\|_{L^2}^2 \int_0^{\sigma(T)} \|\nabla H\|_{L^2}^2 dt + \sup_{\sigma(T) \leq t \leq T} \big(\sigma \|\nabla H\|_{L^2}^2 \big)\int_{\sigma(T)}^T \|\nabla H\|_{L^2}^2 dt \leq C(K) C_0,
\end{align}
which combined with \eqref{3.28} leads to \eqref{3.11} provided that
\begin{align*}
C_0 \leq \varepsilon_1 \triangleq \min\left\{ \varepsilon_{1,1}, \, \frac{1}{2 C_1 C(K)} \right\}. \tag*{\qedhere}
\end{align*}
\end{proof}

Next, we give the estimates for $A_1(T)$ and $A_2(T)$.

\begin{lemma}\label{lem:3.2}
Let the condition \eqref{3.4} be satisfied, then there is a positive constant $\varepsilon_{2}$ depending only on $\mu$, $\lambda$, $\nu$, $A$, $\gamma$, $\hat{\rho}$, $\|\phi\|_{H^2}$, and $\inf\limits_{\Omega}\phi$ such that
\begin{gather}\label{3.30}
A_1(T) \leq C C_0 + C \int_0^T \sigma^2 \|P-P(\rho_s)\|_{L^4}^4 dt,
\\ \label{3.31}
A_2(T) \leq C C_0 + C A_1(T) + C \int_0^T \sigma^2 \|\nabla u\|_{L^4}^4 dt,
\end{gather}
provided $C_0 \leq \varepsilon_{2}$.
\end{lemma}
\begin{proof}
Multiplying \eqref{3.12} by \(\sigma(t) \) and integrating the resulting equation over \((0, T)\), we obtain from \eqref{3.4}, \eqref{3.7}, \eqref{I_1}, \eqref{I_2}, \eqref{I_3}--\eqref{I_5}, and \eqref{I_0} that
\begin{align*}
A_1(T) &\leq C \sigma(\|\nabla u_0\|_{L^2}^2+ \|\nabla H_0\|_{L^2}^2) + C \sigma'\int_0^T\bigg(\f12\mu\|\curl u\|_{L^2}^2+ \f12(2\mu +\lambda)\|\divv u\|_{L^2}^2 + \nu\|\curl H\|_{L^2}^2\bigg)dt \notag\\
&\quad + C\int_{0}^{T}\int\sigma'\bigg(\f12|H|^2\divv u- H\cdot\nabla u\cdot H+ (P-P(\rho_s))\divv u+ (\rho-\rho_s)\nabla\phi\cdot u\bigg)dx dt\notag\\
&\quad + C\sigma \sup_{0 \leq t \leq T} I_0+ C\int_{0}^{T} \sigma\sum_{i=1}^5 I_i dt \notag\\
&\leq CC_0^{\f12}\big(1+\sup_{0 \leq t \leq T}\|\nabla H\|_{L^2}\sup_{0 \leq t \leq T}\big(\sigma\|\nabla H\|_{L^2}^2\big)\big)+ C\int_{0}^{T}\sigma\big(\|\nabla u\|_{L^2}^6+\|\nabla u\||_{L^2}^3+ \|\nabla u\|_{L^2}^2\notag\\
&\quad + \|\nabla u\|_{L^2}^4\|\nabla H\|_{L^2}^2+ \|\nabla u\|_{L^2}\|\nabla H\|_{L^2}^2+ \|\nabla H\|_{L^2}^2+ \|\nabla u\|_{L^2}^2\|\nabla H\|_{L^2}^2 + \|\nabla H\|_{L^2}^6\big)dt\notag\\
&\quad+ C\int_{0}^{T}\sigma^2\|P-P(\rho_s)\|_{L^4}^4 dt + C\int_{0}^{\sigma(T)}\sigma'\big(\|\nabla u\|_{L^2}^2+ \|\nabla H\|_{L^2}^2+ \|\rho-\rho_s\|_{L^2}\big) dt\notag\\
&\leq CC_0^{\f12}+ C\int_{0}^{T}\sigma^2\|P-P(\rho_s)\|_{L^4}^4 dt.
\end{align*}

It is noted that $\eqref{a1}_2$ can be rewritten as
\begin{align}\label{1234}
\rho\dot u= \nabla F- \mu\curl\curl u+ H\cdot\nabla H+ (\rho-\rho_s)\nabla\phi.
\end{align}
Operating \(\sigma^2 \dot{u}_j [\partial/\partial t + \divv(u \cdot)]\) to $\eqref{1234}^j$, summing with respect to \(j\), and integrating over \(\Omega\), we get that
\begin{align}\label{3.34}
&\frac{d}{dt} \int\frac{\sigma^2}{2} \rho |\dot{u}|^2 dx -  \int \sigma \sigma' \rho |\dot{u}|^2 dx \notag \\
& = \int \sigma^2 \left( \dot{u} \cdot \nabla F_t + \dot{u}^j \divv(u \partial_j F) \right) dx+ \mu \int \sigma^2 \left( -\dot{u} \cdot \nabla \times \text{curl} u_t - \dot{u}^j  \divv((\nabla \times \text{curl} u)^j u) \right) dx \notag\\
&\quad+ \int \sigma^2(\dot u^j\partial_t(H^i\partial_i H^j))+ \dot u^j\divv (H^i\partial_i H^j u)) dx+ \int \sigma^2 \left( \rho_t \dot{u} \cdot \nabla \phi + \dot{u}^j
\divv((\rho - \rho_s) \partial_j \psi u) \right) dx
\notag\\
& \triangleq J_1 + \mu J_2 + J_3 + J_4.
\end{align}
For \( J_1 \), by \eqref{a5}, Lemma \ref{lem:2.6}, and \eqref{3.16}, a direct computation yields that
\begin{align}\label{3.35}
J_1 &= \int_{\partial\Omega} \sigma^2 F_t \dot{u} \cdot n  dS - \int \sigma^m F_t  \divv\dot{u}  dx - \int \sigma^2 u \cdot \nabla\dot{u}^j \partial_j F dx \notag\\
&= \int_{\partial\Omega} \sigma^2 F_t \dot{u} \cdot n  dS - (\lambda + 2\mu) \int \sigma^2 (\divv\dot{u})^2  dx + (\lambda + 2\mu) \int \sigma^2 \divv\dot{u}\partial_iu^j\partial_ju^idx \notag\\
&\quad - \gamma \int \sigma^2 P \divv\dot{u} \divv u \, dx + \int \sigma^2 \divv\dot{u}  u \cdot \nabla F \, dx - \int \sigma^2 u \cdot \nabla\dot{u}^j  \partial_j F  dx\notag \\
&\quad + \int \sigma^2\divv\dot{u}  H \cdot H_t  dx + \f12\int \sigma^2 \divv\dot{u}  u \cdot \nabla |H|^2   dx + \int \sigma^2 u\cdot\nabla P(\rho_s)\divv\dot u dx\notag\\
&\leq \int_{\partial\Omega} \sigma^2 F_t \dot{u} \cdot n  dS - (\lambda + 2\mu) \int \sigma^2 (\divv\dot{u})^2  dx + \frac{\delta}{12} \sigma^2 \|\nabla\dot{u}\|_{L^2}^2 + C \sigma^2 (\|\nabla u\|_{L^4}^4+ \|\nabla u\|_{L^2}^2)\notag\\
&\quad + C \sigma^2\big(\|\nabla u\|_{L^2}^2 \|\nabla F\|_{L^3}^2 + C_0^{\f14} \|\nabla H_t\|_{L^2}^2 + \|\nabla u\|_{L^2}^2 \|\nabla H\|_{L^2}^2 \|\nabla H\|_{L^6}^2\big),
\end{align}
where in the third inequality we have used
\begin{align*}
F_t &= (2\mu + \lambda)\divv u_t - P_t- H\cdot H_t \notag\\
&= (2\mu + \lambda)\divv\dot{u}-(2\mu + \lambda)\divv(u \cdot \nabla u) + u \cdot \nabla P + \gamma P \divv u- H\cdot H_t \notag\\
&= (2\mu + \lambda)\divv\dot{u}-(2\mu + \lambda)\partial_iu^j\partial_ju^i - u \cdot \nabla F + \gamma P\divv u+ u\cdot\nabla P(\rho_s)- \f12u\cdot \nabla|H|^2- H\cdot H_t.
\end{align*}
For the first term on the right-hand side of \eqref{3.35}, we get that
\begin{align}\label{3.37}
&\int_{\partial\Omega} \sigma^2 F_t \dot{u} \cdot n dS\notag\\
&  = - \int_{\partial\Omega} \sigma^2 F_t (u \cdot \nabla n \cdot u) dS\notag\\
& = - \Big(\int_{\partial\Omega}\sigma^2 (u \cdot \nabla n \cdot u) F dS\Big)_t+ 2 \sigma \sigma' \int_{\partial\Omega} (u \cdot \nabla n \cdot u) F dS + \sigma^2 \int_{\partial\Omega} F \dot{u} \cdot \nabla n \cdot u dS \notag\\
&\quad+ \sigma^2 \int_{\partial\Omega} F u \cdot \nabla n \cdot \dot{u} dS - \sigma^2 \int_{\partial\Omega} F (u \cdot \nabla) u \cdot \nabla n \cdot u  dS - \sigma^2 \int_{\partial\Omega} F u \cdot \nabla n \cdot (u \cdot \nabla) u dS \notag\\
&\leq - \left( \int_{\partial\Omega} \sigma^2 (u \cdot \nabla n \cdot u) F dS \right)_t + C \sigma \sigma' \|\nabla u\|_{L^2}^2 \|F\|_{H^1} + \frac{\delta}{12} \sigma^2 \|\nabla\dot{u}\|_{L^2}^2 + C \sigma^2 (\|\nabla u\|_{L^2}^4 +  \|\nabla u\|_{L^2}^2 \|F\|_{H^1}^2) \notag\\
&\quad + C\sigma^2 \|\nabla F\|_{L^6} \|\nabla u\|_{L^2}^3 + C\sigma^2 \|F\|_{H^1} \|\nabla u\|_{L^2} \left( \|\nabla u\|_{L^4}^2 + \|\nabla u\|_{L^2}^2 \right),
\end{align}
where we have utilized
\begin{align*}
\Big|\int_{\partial\Omega}(u\cdot\nabla n\cdot u)F dS\Big| \leq C\|\nabla u\|_{L^2}^2 \|F\|_{H^1}.
\end{align*}
Due to
\begin{align*}
u = -(u \times n) \times n = u^{\perp} \times n \ \text{ on } \partial \Omega
\end{align*}
and
\begin{align*}
\divv\big(\nabla u^{i} \times u^{\perp}\big)= u^{\perp}\cdot\curl\nabla u^i-\nabla u^{i} \cdot \curl u^{\perp}= -\nabla u^{i} \cdot \curl u^{\perp},
\end{align*}
where $u^{\perp} \triangleq -u \times n$. Thus, we obtain that
\begin{align}\label{3.40}
- \int_{\partial \Omega} F(u \cdot \nabla) u \cdot \nabla n \cdot u dS
= & -\int_{\partial \Omega} n \cdot \big(\nabla u^{i} \times u^{\perp}\big) \nabla_{i} n \cdot Fu dS \notag\\
= & -\int \divv\big(\big(\nabla u^{\perp} \times u^{\perp}\big) \nabla_{i} n \cdot Fu \big) dx \notag\\
= & -\int \nabla\left(\nabla_{i} n \cdot Fu \right) \cdot \big(\nabla u^{i} \times u^{\perp}\big) dx + \int \nabla u^{i} \cdot \curl u^{\perp} \nabla_{i} n \cdot Fu  dx \notag\\
\leq & C \int |\nabla F||\nabla u||u|^2 dx + C \int |F|\left(|\nabla u|^{2}|u| + |\nabla u||u|^{2}\right) dx \notag\\
\leq &C\|\nabla F\|_{L^6}\|\nabla u\|_{L^2}^3+ C\|F\|_{H^1}\|\nabla u\|_{L^2}\big(\|\nabla u\|_{L^4}^4+ \|\nabla u\|_{L^2}^2\big),
\end{align}
Similarly, we can get
\begin{align}\label{3.41}
-\int_{\partial\Omega} Fu\cdot\nabla n\cdot(u\cdot\nabla)u dS
\leq C\|\nabla F\|_{L^6}\|\nabla u\|_{L^2}^3+ C\|F\|_{H^1}\|\nabla u\|_{L^2}
\big(\|\nabla u\|_{L^4}^2+ \|\nabla u\|_{L^2}^2\big),
\end{align}
which together with Lemma \ref{lem:2.7}, \eqref{3.35}, \eqref{3.37}, \eqref{3.7}, and \eqref{3.40} leads to
\begin{align}\label{J_1}
J_1 &\leq - (\lambda + 2\mu) \int \sigma^2 (\divv\dot{u})^2 \, dx - \left( \int_{\partial\Omega} \sigma^2 (u \cdot \nabla n \cdot u) F \, dS \right)_t \notag\\
&\quad + \frac{\delta}{2} \sigma^2 \|\nabla\dot{u}\|_{L^2}^2 + C \sigma^2 C_0^{\f14} \|\nabla H_t\|_{L^2}^2 + C \sigma^2\|\nabla u\|_{L^2}^4 \|\curl^2 H\|_{L^2}^2 \notag\\
&\quad + C \sigma^2 \|\nabla u\|_{L^4}^4 + C \sigma^2 \left( \|\sqrt{\rho} \dot{u}\|_{L^2}^2 + \|\curl^2 H\|_{L^2}^2 \right) \|\nabla u\|_{L^2}^2 + C \sigma^2 \|\nabla H\|_{L^2}^2 \|\nabla u\|_{L^2}^6 \notag\\
&\quad + C \sigma^2 \|\nabla u\|_{L^2}^2 \left( 1 + \|\nabla u\|_{L^2}^2 + \|\nabla u\|_{L^2}^4 + \|\nabla H\|_{L^2}^2 + \|\nabla H\|_{L^2}^4 + \|\nabla H\|_{L^2}^6 \right) \notag\\
&\quad + C \sigma\sigma'\big(\|\sqrt{\rho}\dot u\|_{L^2}^2+ \|\curl^2 H\|_{L^2}^2\big)\|\nabla u\|_{L^2}^2+ C\|\nabla u\|_{L^2}^2\big(1+ \|\nabla u\|_{L^2}^2+ \|\nabla H\|_{L^2}^2\big),
\end{align}
where the following estimates are employed
\begin{align*}
\|\nabla F\|_{L^3}^2\|\nabla u\|_{L^2}^2 &\leq C(\|\sqrt{\rho}\dot u\|_{L^3}^2+ \|\rho- \rho_s\|_{L^6}^2+ \|H\|_{L^6}^2\|\nabla H\|_{L^6}^2)\|\nabla u\|_{L^2}^2\notag\\
&\leq C\big(\|\sqrt{\rho}\dot u\|_{L^2}\|\sqrt{\rho}\dot u\|_{L^6}+ \|\rho- \rho_s\|_{L^6}^2+ \|\curl^2 H\|_{L^2}^2\|\nabla H\|_{L^2}^2+ \|\nabla H\|_{L^2}^4\big)\|\nabla u\|_{L^2}^2\notag\\
&\leq C\big(\|\sqrt{\rho}\dot u\|_{L^2}\|\nabla\dot u\|_{L^2}+ \|\sqrt{\rho}\dot u\|_{L^2}\|\nabla u\|_{L^2}^2+ \|\rho- \rho_s\|_{L^6}^2\notag\\
&\quad + \|\curl^2 H\|_{L^2}^2\|\nabla H\|_{L^2}^2+ \|\nabla H\|_{L^2}^4\big)\|\nabla u\|_{L^2}^2\notag\\
&\leq \frac{\delta}{12}\|\nabla\dot u\|_{L^2}^2+ C\big(\|\sqrt{\rho}\dot u\|_{L^2}^2+\|\curl^2 H\|_{L^2}^2\big)\|\nabla u\|_{L^2}^2\notag\\
&\quad + C\|\nabla u\|_{L^2}^2\big(1+ \|\nabla u\|_{L^2}^2+ \|\nabla u\|_{L^2}^4+ \|\nabla H\|_{L^2}^4\big),
\\
\|\nabla u\|_{L^2}^2\|F\|_{H^1}^2 &\leq C\|\nabla u\|_{L^2}^2\big(\|\sqrt{\rho}\dot u\|_{L^2}^2+ \|\rho-\rho_s\|_{L^3}^2+ \|\curl^2 H\|_{L^2}^2\|\nabla H\|_{L^2}^2+ \|\nabla H\|_{L^2}^4+ \|\nabla u\|_{L^2}^2\big)\notag\\
&\leq C\big(\|\sqrt{\rho}\dot u\|_{L^2}^2+ \|\curl^2 H\|_{L^2}^2\big)\|\nabla u\|_{L^2}^2+ C\|\nabla u\|_{L^2}^2\big(1+ \|\nabla u\|_{L^2}^2+ \|\nabla H\|_{L^2}^4\big),
\\
\|\nabla F\|_{L^6}\|\nabla u\|_{L^2}^3 &\leq C\Big(\|\nabla\dot u\|_{L^2}+ \|\nabla u\|_{L^2}^2+ \|\rho- \rho_s\|_{L^{\infty}}+ \|\nabla H\|_{L^2}^{\f12}\big(\|\curl^2 H\|_{L^2}+ \|\nabla H\|_{L^2}\big)^{\f32}\Big)\|\nabla u\|_{L^2}^3\notag\\
&\leq \frac{\delta}{12}\|\nabla\dot u\|_{L^2}^2+ C\|\nabla u\|_{L^2}^4\|\curl^2 H\|_{L^2}^2+ C\|\nabla H\|_{L^2}^2+ C\|\nabla H\|_{L^2}^2\|\nabla u\|_{L^2}^6\notag\\
&\quad + C\|\nabla u\|_{L^2}^2\big(1+ \|\nabla u\|_{L^2}^2+ \|\nabla u\|_{L^2}^4\big)+ C\|\nabla H\|_{L^2}^2,
\end{align*}
and
\begin{align*}
\|\nabla u\|_{L^2}^2\|\nabla H\|_{L^2}^2\|\nabla H\|_{L^6}^2 &\leq C\|\nabla u\|_{L^2}^2\|\nabla H\|_{L^2}^2\big(\|\curl^2 H\|_{L^2}^2+ \|\nabla H\|_{L^2}^2\big)\notag\\
&\leq C\|\curl^2 H\|_{L^2}^2\|\nabla u\|_{L^2}^2+ C\|\nabla u\|_{L^2}^2\|\nabla H\|_{L^2}^6+ C\|\nabla u\|_{L^2}^2\|\nabla H\|_{L^2}^2.
\end{align*}

Observing that \(\text{curl}\,u_t = \text{curl}\,\dot{u} - u \cdot \nabla\text{curl}\,u - \nabla u^i \times \nabla_i u\), we get that
\begin{align}\label{J_2}
J_2 &= - \int \sigma^2 |\text{curl}\,\dot{u}|^2 \, dx + \int \sigma^2 \text{curl}\,\dot{u} \cdot (\nabla u^i \times \nabla_i u) \, dx \notag\\
&\quad + \int \sigma^2 u \cdot \nabla\text{curl}\,u \cdot \text{curl}\,\dot{u} \, dx + \int \sigma^2 u \cdot \nabla\dot{u} \cdot (\nabla \times \text{curl}\,u) \, dx \notag\\
&\leq - \int \sigma^2 |\text{curl}\,\dot{u}|^2 \, dx + \frac{\delta}{12}\sigma^2\|\nabla\dot u\|_{L^2}^2+ C\sigma^2\|\nabla u\|_{L^4}^4.
\end{align}
A direct computation shows that
\begin{align}\label{J_3}
J_3 &= - \sigma^2 \int \nabla \dot{u} : (H \otimes H)_t \, dx - \mu \int \sigma^2 H \cdot \nabla H^j u \cdot \nabla \dot{u}^j \, dx \notag\\
&\leq C \sigma^2\big( \| \nabla \dot{u} \|_{L^2} \| H \|_{L^3} \| H_t \|_{L^6} + \| \nabla \dot{u} \|_{L^2} \| H \|_{L^6} \| \nabla H \|_{L^6} \| u \|_{L^6} \big) \notag\\
&\leq \frac{\delta}{12} \sigma^2 \| \nabla \dot{u} \|_{L^2}^2 + C \sigma^2 \big( \| \nabla H \|_{L^2}^4 + \| \nabla u \|_{L^2}^4 \big) \| \curl^2 H \|_{L^2}^2 \notag\\
&\quad + C \sigma^2 C_{0}^{\f12} \| \nabla H_t \|_{L^2}^2 + C \sigma^2 \| \nabla H \|_{L^2}^4 \| \nabla u \|_{L^2}^2,
\\ \label{J_4}
J_4&= \int \sigma^2\big(\rho u\cdot\nabla(\dot u\cdot\nabla\phi)-(\rho-\rho_s)(u\cdot\nabla\dot u)\cdot\nabla\phi\big)dx \notag\\
&\leq \sigma^2\big(\|u\|_{L^3}\|\nabla\dot u\|_{L^2}\|\nabla\phi\|_{L^6}+ \|u\|_{L^3}\|\dot u\|_{L^6}\|\nabla^2\phi\|_{L^2}\big)\notag\\
&\leq \frac{\delta}{12}\sigma^2\|\nabla\dot u\|_{L^2}^2+ C(\delta)\sigma^2\big(\|\nabla u\|_{L^2}^2 +\|\nabla u\|_{L^2}^4\big).
\end{align}

Consequently, substituting \eqref{J_1}--\eqref{J_4} into \eqref{3.34} gives that
\begin{align}\label{3.46}
&\Big(\f{\sigma^2}{2}\|\sqrt{\rho}\dot u\|_{L^2}^2\Big)_t +(2\mu+ \lambda)\sigma^2\|\divv\dot u\|_{L^2}^2+ \mu\sigma^2\|\curl\dot u\|_{L^2}^2 \notag\\
&\leq -\bigg(\int_{\partial\Omega}\sigma^2(u\cdot\nabla n\cdot u)F dS\bigg)_t+ \delta\sigma^2\|\nabla\dot u\|_{L^2}^2+ C\sigma^2C_0^{\f12}\|\nabla H_t\|_{L^2}^2+ C\sigma^2\|\nabla u\|_{L^4}^4\notag\\
&\quad + C\sigma^2\big(\|\nabla H\|_{L^2}^4+ \|\nabla u\|_{L^2}^4\big)\|\curl^2 H\|_{L^2}^2+ C\sigma^2\big(\|\sqrt{\rho}\dot u\|_{L^2}^2+ \|\curl^2 H\|_{L^2}^2\big)\|\nabla u\|_{L^2}^2\notag\\
&\quad + C\sigma^2\|\nabla H\|_{L^2}^2\|\nabla u\|_{L^2}^6+ C\sigma^2\|\nabla u\|_{L^2}^2\big(1+\|\nabla u\|_{L^2}^2+\|\nabla u\|_{L^2}^4+\|\nabla H\|_{L^2}^2 +\|\nabla H\|_{L^2}^4+C\|\nabla H\|_{L^2}^6\big)\notag\\
&\quad+ C\sigma\sigma'\big(\|\sqrt{\rho}\dot u\|_{L^2}^2+ \|\curl^2 H\|_{L^2}^2+ \|\nabla u\|_{L^2}^2 + \|\nabla u\|_{L^2}^4+\|\nabla H\|_{L^2}^2\|\nabla u\|_{L^2}^2\big).
\end{align}
By \eqref{3.1} and \eqref{3.4}, choosing $\delta$ small enough, and integrating \eqref{3.46} over $(0,T)$, we get that
\begin{align}\label{3.47}
&\sup_{0 \leq t \leq T}\big(\sigma^2\|\sqrt{\rho}\dot{u}\|_{L^2}^2\big)+ \int_0^T \sigma^2 \|\nabla \dot{u}\|_{L^2}^2 dt\notag \\
&\leq - \sup_{0 \leq t \leq T}\int_{\partial \Omega} \sigma^2 (u \cdot \nabla n \cdot u) F dS + C C_0^\f12 \int_0^T \sigma^2 \|\nabla H_t\|_{L^2}^2 dt + C \int_0^T \sigma^2 \|\nabla u\|_{L^4}^4 dt\notag\\
&\quad + C C_0 \sup_{0 \leq t \leq T} \big[\sigma^2 \big(\|\sqrt{\rho}\dot u\|_{L^2}^2+ \|\curl^2 H\|_{L^2}^2 \big)\big]+ C C_0 \sup_{0 \leq t \leq T}\big[ \sigma \left( \|\nabla u\|_{L^2}^2 \|\nabla H\|_{L^2}^2 + \|\nabla u\|_{L^2}^2 \right)\big] \notag\\
&\quad + C C_0 \sup_{0 \leq t \leq \sigma(T)}\big(\sigma \|\nabla u\|_{L^2}^2\big) + C \int_0^{\sigma(T)} \sigma \left( \|\sqrt{\rho}\dot{u}\|_{L^2}^2 + \|\curl^2 H\|_{L^2}^2 \right) \, dt +CC_0^{\f12} \notag\\
&\leq -\sup_{0 \leq t \leq T}\int_{\partial \Omega} \sigma^2(u\cdot\nabla n\cdot u)F dS+ CC_0^{\f12}\Big(1+\int _0^T \sigma^2 \|\nabla H_t\|_{L^2}^2 dt\Big)
+C A_{1}(T)+ C\int _0^T \sigma^2\|\nabla u\|_{L^4}^4 dt.
\end{align}
For the boundary term on the right-hand side of \eqref{3.47}, one infers from Lemma \ref{lem:2.6} that
\begin{align*}
\int_{\partial \Omega} (u \cdot \nabla n \cdot u) F dS
\leq C \|\nabla u\|_{L^2}^2 \|F\|_{H^1}
\leq \f14\|\sqrt{\rho}\dot u\|_{L^2}^2 + \f14\|\curl^2 H\|_{L^2}^2 + C \left( \|\nabla u\|_{L^2}^2 + \|\nabla u\|_{L^2}^4 \right).
\end{align*}
Therefore, it follows that
\begin{align}\label{3.49}
&\sup_{0 \leq t \leq T}\big(\sigma^2\|\sqrt{\rho}\dot{u}\|_{L^2}^2\big)+ \int_0^T \sigma^2 \|\nabla \dot{u}\|_{L^2}^2 \, dt -CC_0^{\f12}\int_0^T \sigma^2\|\nabla H_t\|_{L^2}^2 dt\notag\\
&\leq \frac{\sigma^2}{4}\sup_{0 \leq t \leq T}\|\curl^2 H\|_{L^2}^2+ CC_0^{\f12}+ CC_0^{\f12}\int_0^T \sigma^2\|\nabla H_t\|_{L^2}^2 dt + CA_{1}(T)+ C\int_0^T\sigma^2\|\nabla u\|_{L^4}^4 dt
\end{align}
due to \eqref{3.4}.

Next, we need to estimate the term $\|\nabla H_t\|_{L^2}$. Noticing that
\begin{align*}
\begin{cases}
H_{tt} + \nu \nabla \times (\curl H_t) = (H \cdot \nabla u - u \cdot \nabla H - H \divv u)_t, \quad &\text{in } \Omega, \\
H_t \cdot n = 0, \quad \curl H_t \times n = 0, \quad &\text{on } \partial \Omega,
\end{cases}
\end{align*}
we obtain that
\begin{align}\label{3.51}
& \frac{1}{2} \frac{d}{dt}\sigma^2 \|H_t\|_{L^2}^2 + \sigma^2 \|\text{curl} H_t\|_{L^2}^2 -\sigma \sigma' \|H_t\|_{L^2}^2 \, dx \notag\\
&= \int \sigma^2 (H_t \cdot \nabla u - u \cdot \nabla H_t - H_t \divv u) \cdot H_t dx+ \int \sigma^2 (H \cdot \nabla \dot{u} - \dot{u} \cdot \nabla H - H \divv \dot{u}) \cdot H_t dx \notag\\
&\quad- \int \sigma^2 (H \cdot \nabla (u \cdot \nabla u) - (u \cdot \nabla u) \cdot \nabla H - H \divv (u \cdot \nabla u)) \cdot H_t dx \triangleq K_1 + K_2 + K_3 .
\end{align}
By Lemmas \ref{lem:2.1} and \ref{lem:2.6}, a direct calculation leads to
\begin{align}\label{K_1}
K_1 &\leq C\sigma^2 \big(\|H_t\|_{L^3}\|H_t\|_{L^6}\|\nabla u\|_{L^2}+ \|u\|_{L^6}\|H_t\|_{L^3}\|\nabla H_t\|_{L^2}\big)\notag\\
&\leq \f18\sigma^2\|\nabla H_t\|_{L^2}^2+ C\sigma^2\|\nabla u\|_{L^2}^4\|H_t\|_{L^2}^2.
\end{align}
Similarly, one has that
\begin{align}\label{K_2}
K_2 &\leq C \sigma^2 \|H\|_{L^3} \|H_t\|_{L^6} \|\nabla \dot{u}\|_{L^2} - \int_{\partial \Omega} \sigma^2 (\dot{u} \cdot n) (H \cdot H_t) \, dS + \int \sigma^2 \dot{u} \cdot \nabla H_t \cdot H \, dx \notag\\
&\leq \int_{\partial \Omega} \sigma^2 (u \cdot \nabla n \cdot u) (H \cdot H_t) \, dS +\f18\sigma^2\|\nabla H_t\|_{L^2}^2+ CC_{0}^{\f12} \sigma^2 \left( \|\nabla \dot{u}\|_{L^2}^2 + \|\nabla u\|_{L^2}^4  \right)
\end{align}
due to
\begin{align*}
\int \sigma^2 \dot u\cdot\nabla H\cdot H_t dx =\int_{\partial \Omega}\sigma^2(\dot u\cdot n)(H\cdot H_t) dS- \int \sigma^2\divv\dot u H\cdot H_t dx -\int \dot u\cdot\nabla H_t\cdot H dx.
\end{align*}
By the trace theorem and Lemma \ref{lem:2.6}, it follows that
\begin{align*}
\int_{\partial\Omega}\sigma^2(u\cdot\nabla n\cdot u)(H\cdot H_t) dS
&\leq C\sigma\big(\|u\|_{L^6}\|\nabla u\|_{L^2}\|H\|_{L^6}\|H_t\|_{L^6}+ \|u\|_{L^6}^2\|\nabla H\|_{L^2}\|H_t\|_{L^6}\notag\\
&\quad + \|u\|_{L^6}^2\|\nabla H_t\|_{L^2}\|H\|_{L^6}+ \|u\|_{L^4}^2\|H\|_{L^3}\|H_t\|_{L^6}\big)\notag\\
&\leq \f18\sigma^2\|\nabla H_t\|_{L^2}^2+ C\sigma^2\|\nabla u\|_{L^2}^4\|\nabla H\|_{L^2}^2.
\end{align*}
Using similar arguments, one has that
\begin{align}\label{K_3}
K_3 &= \int \sigma^2 H \cdot \nabla H_t \cdot (u \cdot \nabla u) dx + \int_{\partial \Omega} \sigma^2 H \cdot H_t (u \cdot \nabla u \cdot n) dS
- \int \sigma^2 u \cdot \nabla u \cdot \nabla H_t \cdot H dx \notag\\
&\leq \f14 \sigma^2 \|\nabla H_t\|_{L^2}^2 + C \sigma^2 \left( \|\nabla u\|_{L^2}^4 + \|\nabla H\|_{L^2}^4 \right) \big( \|\sqrt{\rho}\dot{u}\|_{L^2}^2 + \|\curl^2 H\|_{L^2}^2 \big)\notag \\
&\quad +C\sigma^2\|\nabla u\|_{L^2}^2\|\nabla H\|_{L^2}^2\big(\|\nabla u\|_{L^2}^2+ \|\nabla H\|_{L^2}^2+1\big).
\end{align}

Putting \eqref{K_1}--\eqref{K_3} into \eqref{3.51} and choosing $\delta$ suitably small, we have
\begin{align*}
&\frac{d}{dt}\left( \sigma^2 \|H_t\|_{L^2}^2 \right) + \sigma^2 \|\nabla H_t\|_{L^2}^2 - C C_{0}^{\f12} \sigma^2  \|\nabla \dot{u}\|_{L^2}^2 \notag\\
&\leq C \sigma^2 \left( \|\nabla u\|_{L^2}^4 + \|\nabla H\|_{L^2}^4 \right) \big( \|\sqrt{\rho}\dot{u}\|_{L^2}^2 + \|\curl^2 H\|_{L^2}^2 + \|H_t\|_{L^2}^2 \big)\notag \\
&\quad + C \sigma^2 \|\nabla u\|_{L^2}^2 \|\nabla H\|_{L^2}^2 \left( \|\nabla u\|_{L^2}^2 + \|\nabla H\|_{L^2}^2 + 1 \right) +C\sigma^2 \|\nabla u\|_{L^2}^4+  C \sigma \sigma' \|H_t\|_{L^2}^2.
\end{align*}
Integrating the above inequality over $(0,T)$ indicates that
\begin{align}\label{3.59}
&\sup_{0 \leq t \leq T}\big(\sigma^2\|H_t\|_{L^2}^2\big)+ \int_0^T \sigma^2 \|\nabla H_t\|_{L^2}^2 \, dt - C_3 C_0^{\f12} \int_0^T \sigma^2 \|\nabla \dot{u}\|_{L^2}^2 \, dt \notag\\
&\leq C C_{0} \sup_{0 \leq t \leq T} \sigma^2 \left( \|\sqrt{\rho}\dot{u}\|_{L^2}^2 + \|\curl^2 H\|_{L^2}^2 + \|H_t\|_{L^2}^2 \right) \notag\\
&\quad + C C_0 \sup_{0 \leq t \leq T} \sigma \left( \|\nabla u\|_{L^2}^2 + \|\nabla H\|_{L^2}^2 \right) + C \int_0^{\sigma(T)}  \sigma\sigma' \|H_t\|_{L^2}^2 dt.
\end{align}
This together with \eqref{3.49} implies that
\begin{align}\label{3.60}
& \sup_{0 \leq t \leq T}\big[\sigma^2 \big(\|\sqrt{\rho}\dot{u}\|_{L^2}^2 + \|H_t\|_{L^2}^2 \big)\big]+ \int_0^T \sigma^2 \left( \|\nabla \dot{u}\|_{L^2}^2 + \|\nabla H_t\|_{L^2}^2 \right) \, dt \notag\\
& \quad- C_2 C_0^{\f12} \int_0^T \sigma^2 \|\nabla H_t\|_{L^2}^2 \, dt - C_3 C_0^{\f12} \int_0^T \sigma^2\|\nabla \dot{u}\|_{L^2}^2 \, dt \notag\\
&\leq C\int_0^T \sigma^2\|\nabla H_t\|_{L^2}^2 dt+ CA_{1}(T)+ C\int_0^{\sigma(T)}\sigma\sigma'\|H_t\|_{L^2}^2 dt \notag\\
& \leq C\int_0^T \sigma^2\|\nabla u\|_{L^4}^4 dt+ CA_{1}(T)+ CC_0^{\f12}
\end{align}
provided that
\begin{align}
C_0 \leq \varepsilon_{2}\triangleq \min \left\{\varepsilon_1,\Big(\frac{1}{4C_2}\Big)^2,\Big(\frac{1}{4C_3}\Big)^2\right\}.\notag
\end{align}

Finally, by Lemma \ref{lem:2.1} and $\eqref{a1}_3$, we have
\begin{align*}
\|\curl^2 H\|_{L^2}^2 &\leq C\big(\|H_t\|_{L^2}^2+ \|u\nabla H\|_{L^2}^2+ \|H\nabla u\|_{L^2}^2\big)\notag\\
&\leq C\big(\|H_t\|_{L^2}^2+ \|u\|_{L^6}^2\|\nabla H\|_{L^3}^2+ \|H\|_{L^\infty}^2\|\nabla u\|_{L^2}^2\big)\notag\\
&\leq C\big(\|H_t\|_{L^2}^2+ \|\nabla u\|_{L^2}^2\big(\|\nabla ^2H\|_{L^2}\|\nabla H\|_{L^2}+ \|\nabla H\|_{L^2}^2\big)\notag\\
&\quad + \|\nabla H\|_{L^2}\big(\|\curl^2 H\|_{L^2}+ \|\nabla H\|_{L^2})\|\nabla u\|_{L^2}^2\big)\notag\\
&\leq \f12\|\curl^2 H\|_{L^2}^2 +C\|H_t\|_{L^2}^2+ C\|\nabla H\|_{L^2}^2\|\nabla u\|_{L^2}^4+ C\|\nabla u\|_{L^2}^2\|\nabla H\|_{L^2}^2 .
\end{align*}
Thus, we have
\begin{align}\label{3.62}
\sup_{0 \leq t \leq T}\big(\sigma^2\|\curl^2 H\|_{L^2}^2\big)
\leq C\sup_{0 \leq t \leq T}\big[\sigma^2\big(\|H_t\|_{L^2}^2+ \|\nabla H\|_{L^2}^2\|\nabla u\|_{L^2}^4+ \|\nabla H\|_{L^2}^2\|\nabla u\|_{L^2}^2\big)\big].
\end{align}
This together with \eqref{3.60} gives that
\begin{align*}
A_2(T) &\leq CC_0+ CA_1(T)+ C\sup_{0 \leq t \leq T}\big((\sigma\|\nabla u\|_{L^2}^2)^2\|\nabla H\|_{L^2}^2\big)\notag\\
&\quad + C\sup_{0 \leq t \leq T}\big(\sigma\|\nabla u\|_{L^2}^2\sigma\|\nabla H\|_{L^2}^2\big)+ C\int_0^T \sigma^2\|\nabla u\|_{L^4}^4 dt,
\end{align*}
which implies \eqref{3.31}.
\end{proof}

\begin{lemma}\label{lem:3.4}
Let the condition \eqref{3.4} be satisfied and $F$ and $\omega$ be the ones defined in \eqref{f5}, then there exists a positive constant $\varepsilon_{3}$ depending only on $\mu$, $\lambda$, $\nu$, $\gamma$, $A$, $\underline\rho$, $\bar\rho$, $\hat{\rho}$, and $\Omega$ such that
\begin{gather}\label{3.64}
\int_{0}^{T} \|\rho - \rho_s\|_{L^2}^2 \, dt \leq C C_0,\\ \label{3.68}
\int_{0}^{T} \sigma^{2}\left(\|\nabla u\|_{L^{4}}^{4} + \|F\|_{L^{4}}^{4} + \|\omega\|_{L^{4}}^{4} + \|P - P(\rho_s)\|_{L^{4}}^{4}\right) d t \leq C C_{0},
\end{gather}
provided $C_0\leq\varepsilon_{3}$.
\end{lemma}
\begin{proof}
Note that $\nabla\phi=\gamma\rho_s^{\gamma- 2}\nabla \rho_s$ due to \eqref{state density},
and the definition of $G(\rho, \rho_s)$ yields that
\begin{align}
\rho_s^{-1}(\nabla(\rho^{\gamma}- \rho_s^{\gamma})- \gamma(\rho- \rho_s)\rho_s^{\gamma- 2}\nabla \rho_s)= \nabla(\rho_s^{-1}(\rho^{\gamma}- \rho_s^{\gamma}))- \frac{\gamma- 1}{A}G(\rho, \rho_s)\nabla\rho_s^{-1}.\notag
\end{align}
Thus, we can rewrite $\eqref{a1}_2$ as
\begin{align}\label{3.65}
- \nabla \left( \rho_s^{-1} \left( P - P(\rho_s) \right) \right)
&= \rho_s^{-1} \bigg( \rho \dot{u} - (\mu+ \lambda) \nabla \divv u - \mu \Delta u -H\cdot\nabla H +\f12\nabla|H|^2\bigg) \notag\\ & \quad - \frac{\gamma - 1}{A} G(\rho, \rho_s)\nabla \rho_s^{-1}.
\end{align}
Multiplying \eqref{3.65} by \(\mathcal{B}[\rho - \rho_s]\) and integrating over \(\Omega\), one gets from Lemma \ref{lem:2.2} that
\begin{align}\label{3.66}
& \int \rho_s^{-1} \left( P - P(\rho_s) \right) (\rho - \rho_s) dx \notag \\
& = \left( \int \rho_s^{-1} \rho u \cdot \mathcal{B}[\rho - \rho_s] dx \right)_t - \int \rho u \cdot \mathcal{B}[\rho_t] dx  - \int \rho_s^{-1} \rho u^i u^j \partial_j B_i(\rho - \rho_s) dx \notag \\
&\quad - \int \rho u^i u^j (\partial_i \rho_s^{-1} B_j(\rho - \rho_s)) dx + \mu \int \rho_s^{-1} \nabla u : \nabla \mathcal{B}(\rho - \rho_s) dx + \mu \int \partial_i u^j \partial_i \rho_s^{-1} B_j(\rho - \rho_s) dx\notag \\
&\quad - \frac{\gamma - 1}{A} \int G(\rho, \rho_s) \nabla\rho_s^{-1}\cdot\mathcal{B}[\rho-\rho_s] dx+ (\lambda + \mu) \int \left( \rho_s^{-1} (\rho - \rho_s) + \nabla \rho_s^{-1} \cdot \mathcal{B}(\rho - \rho_s) \right) \divv u dx  \notag \\
&\quad + \int H\cdot\nabla\rho_s^{-1}\cdot\mathcal{B}[\rho-\rho_s] dx+ \int \rho_s^{-1}H\cdot\nabla \mathcal{B}[\rho-\rho_s]\cdot H dx
\notag \\ & \quad- \f12\int |H|^{2}\mathcal{B}[\rho-\rho_s]\cdot\nabla\rho_s^{-1}  dx- \f12\int\rho_s^{-1}|H|^2\divv(\mathcal{B}[\rho-\rho_s]) dx \notag \\
& \leq \left( \int \rho_s^{-1} \rho u \cdot \mathcal{B}[\rho - \rho_s] \, dx \right)_t + C \|\nabla u\|_{L^2}^2 + C \|u\|_{L^6}^2 \|\nabla \rho_s^{-1}\|_{L^6} \|\mathcal{B}[\rho - \rho_s]\|_{L^2} \notag \\
& \quad + C \|\nabla u\|_{L^2} \|\mathcal{B}[\rho - \rho_s]\|_{L^2}+ C \|\nabla u\|_{L^2} \|\nabla \rho_s^{-1}\|_{L^6} \|\mathcal{B}[\rho - \rho_s]\|_{L^3} + C \|\nabla u\|_{L^2} \|\rho - \rho_s\|_{L^2} \notag \\
&\quad + C \|G(\rho, \rho_s)\|_{L^3}^{\f12} \|G(\rho, \rho_s)\|_{L^1}^{\f12} \|\nabla \rho_s^{-1}\|_{L^6} \|\mathcal{B}[\rho - \rho_s]\|_{L^6}+ C \|H\|_{L^6}^2\|\nabla\rho_s^{-1}\|_{L^6}\|\mathcal{B}[\rho-\rho_s]\|_{L^2} \notag \\
&\quad + C\|H\|_{L^4}^2\|\nabla\mathcal{B}[\rho-\rho_s]\|_{L^2}+ C\|H\|_{L^6}^2\|\nabla\rho_s^{-1}\|_{L^6}\|\mathcal{B}[\rho-\rho_s]\|_{L^2}+ C\|H\|_{L^4}^2\|\nabla\mathcal{B}[\rho-\rho_s]\|_{L^2} \notag \\
&\leq \left( \int \rho_s^{-1} \rho u \cdot \mathcal{B}[\rho - \rho_s] \, dx \right)_t + \delta\|\rho-\rho_s\|_{L^2}^2 + C\|\nabla u\|_{L^2}^2 + C\|\nabla H\|_{L^2}^4 + C_4C_0^{\f16}\|\rho-\rho_s\|_{L^2}^2 .
\end{align}
Now choosing \(\delta = \frac{1}{4}\) and using Lemma \ref{lem:3.1}, we obtain \eqref{3.64} provided that
\begin{align*}
C_0 \leq \varepsilon_{3}\triangleq\min\bigg\{\varepsilon_{2},\Big(\frac{1}{4 C_4} \Big)^6\bigg\},
\end{align*}
where we have used
\begin{align*}
\bigg|\int_0^T \bigg(\int\rho_s^{-1}\rho u\cdot\mathcal{B}[\rho-\rho_s] dx\bigg)_t dt\bigg| \leq \bigg|\int \rho u\cdot\mathcal{B}[\rho-\rho_s] dx\bigg| +\|\sqrt{\rho_0}u_0\|_{L^2}^2\|\mathcal{B}[\rho_0-\rho_s]\|_{L^2} \leq CC_0.
\end{align*}

In view of the standard $L^{p}$-estimate, we have
\begin{align*}
\int_{0}^{T} \sigma^{2}\|\nabla u\|_{L^{4}}^{4} d t &\leq C \int_{0}^{T} \sigma^{2}\big(\|\divv u\|_{L^{4}}^{4} + \|\omega\|_{L^{4}}^{4}\big) d t \\
&\leq C \int_{0}^{T} \sigma^{2}\big(\|F\|_{L^{4}}^{4} + \|\omega\|_{L^{4}}^{4} + \|P - P(\rho_s)\|_{L^{4}}^{4} + \|H\|_{L^{8}}^{8}\big)dt.
\end{align*}
It follows from \eqref{f5}, \eqref{3.4}, Lemma \ref{lem:3.1}, \eqref{3.10}, and \eqref{3.11} that
\begin{equation*}
\|F\|_{L^{2}} + \|\omega\|_{L^{2}} \leq C\big(\|\nabla u\|_{L^{2}} + \|P - P(\rho_s)\|_{L^{2}} + \|H\|_{L^{4}}^{2}\big) \leq C.
\end{equation*}
This along with Lemma \ref{lem:2.1}, \eqref{3.4}, \eqref{3.7}, \eqref{3.11}, and \eqref{3.19} implies that
\begin{align}\label{3.71}
&\int_{0}^{T} \sigma^{2}\big(\|F\|_{L^{4}}^{4} + \|\omega\|_{L^{4}}^{4}\big)dt\notag \\
&\leq C \int_{0}^{T} \big[\sigma^{2}\left(\|F\|_{L^{2}} + \|\omega\|_{L^{2}}\right)\left(\|\nabla F\|_{L^{2}}^{3} + \|\nabla \omega\|_{L^{2}}^{3}\right)+ \sigma^2\big(\|F\|_{L^2}^4 +\|\omega\|_{L^2}^4\big)\big] dt
\notag \\
&\leq C\sup _{0 \leq t \leq T}\big(\sigma\|\sqrt{\rho}\dot{u}\|_{L^{2}}\big) \int_{0}^{T} \sigma\|\sqrt{\rho}\dot{u}\|_{L^{2}}^{2} d t + C\sup _{0 \leq t \leq T} \big(\sigma\|\nabla^{2} H\|_{L^{2}}\big)^{\f32} \int_{0}^{T}\|\nabla H\|_{L^{2}}^{2} d t\notag \\
&\quad + C\int_0^T \|\rho-\rho_s\|_{L^2}^2 dt+ \sup_{0 \leq t \leq T}\big(\sigma^2\|\nabla u\|_{L^2}^2\big)\int_0^T \|\nabla u\|_{L^2}^2 dt + \sup_{0 \leq t \leq T}\big(\sigma^2\|\nabla H\|_{L^2}^2\big)\int_0^T\|\nabla H\|_{L^2}^2 dt \notag \\
&\leq C A_{2}^{\f12}(T) A_{1}(T) + C C_{0} \leq C C_{0}^{\f34},
\end{align}
where we have used
\begin{align*}
\|F\|_{L^2}^4+ \|\omega\|_{L^2}^4 \leq C\big(\|F\|_{L^2}^3+ \|\omega\|_{L^2}^3\big).
\end{align*}
By Lemma \ref{lem:3.4}, \eqref{3.4}, Lemma \ref{lem:3.1}, \eqref{3.71}, and Young's inequality, we get that
\begin{align}\label{3.73}
\int_0^T \sigma^2\|P-P(\rho_s)\|_{L^4}^4 dt \leq C\int_0^T \sigma^2\|\rho-\rho_s\|_{L^2}^2 dt \leq CC_0.
\end{align}
Moreover, one has that
\begin{align*}
\int_{0}^{T}\sigma^{2}\|H\|_{L^8}^8dt
&\leq C \int_{0}^{T} \sigma^{2} \|H\|_{L^{\infty}}^4 \|H\|_{L^4}^4 dt \notag \\
&\leq C \int_{0}^{T} \sigma^{2}\Big(\|H\|_{L^2} \|\nabla H\|_{L^2}^5 \|\curl^2 H\|_{L^2}^2 +\|H\|_{L^2}\|\nabla H\|_{L^2}^7\Big)dt\notag\\
&\leq C C_0^{\f12} \int_{0}^{T} \sigma^{2} \|\curl^2 H\|_{L^2}^2 dt+CC_0 \leq C C_0^{\f34},
\end{align*}
which combined with \eqref{3.71} and \eqref{3.73} leads to \eqref{3.68}.
\end{proof}

\begin{lemma}
Let the condition \eqref{3.4} be satisfied, then there exists a positive constant $\varepsilon_{4}$ depending on $\mu$, $\lambda$, $\nu$, $\gamma$, $A$, $\hat{\rho}$, $\|\phi\|_{H^2},\inf\limits_{\Omega}\phi$, $\Omega$, $M_1$, and $M_2$ such that
\begin{gather}\label{3.75}
A_1(T) + A_2(T) \leq C_0^{\f12},\\ \label{3.76}
\sup_{0 \leq t \leq T}\big(\|\nabla u\|_{L^{2}}^{2}+\|\nabla H\|_{L^{2}}^{2}\big)+\int_{0}^{T}\big(\|\sqrt{\rho} \dot{u}\|_{L^{2}}^{2}+\left\|H_{t}\right\|_{L^{2}}^{2}+\left\|\nabla^{2} H\right\|_{L^{2}}^{2}\big) d t \leq C,\\ \label{3.77}
\sup_{0 \leq t \leq T}\big[\sigma\big(\|\sqrt{\rho} \dot{u}\|_{L^{2}}^{2}+\left\|H_{t}\right\|_{L^{2}}^{2}+\left\|\nabla^{2} H\right\|_{L^{2}}^{2}\big)\big]+\int_{0}^{T}\sigma\big(\|\nabla \dot{u}\|_{L^{2}}^{2}+\left\|\nabla H_{t}\right\|_{L^{2}}^{2}\big)dt\leq C,
\end{gather}
provided \(C_{0} \leq \varepsilon_{4}\).
\end{lemma}
\begin{proof}
It follows from \eqref{3.30}, \eqref{3.31}, and \eqref{3.68} that
\begin{equation*}
A_1(T) + A_2(T) \leq C C_0^{\f34}.
\end{equation*}
Thus, one immediately obtains \eqref{3.75} provided that
\begin{align*}
C_0 \leq \varepsilon_{4} \triangleq \min\big\{\varepsilon_{3},C^{-4}\big\}.
\end{align*}

By Lemma \ref{lem:2} and \eqref{3.75}, we can get \eqref{3.76} directly.

To prove \eqref{3.77}, operating $\sigma\dot u[\partial_t + \divv(u\cdot)]$ to both sides of $\eqref{a1}_2$, then we deduce from \eqref{3.7}, \eqref{3.75}, \eqref{3.76}, and \eqref{3.46} (with $\sigma^2$ replaced by $\sigma$) that
\begin{align}\label{qwe}
&\big(\sigma\|\sqrt{\rho}\dot u\|_{L^2}^2\big)_t + \sigma\|\nabla\dot u\|_{L^2}^2 \leq -\bigg(\int_{\partial\Omega}\sigma(u\cdot n\cdot u)F dS\bigg)_t +\frac{\sigma}{2}\|\nabla\dot u\|_{L^2}^2+ C_2\sigma C_0^{\f12}\|\nabla H_t\|_{L^2}^2 \notag\\
&\quad + C\sigma\|\nabla u\|_{L^4}^4+ C\sigma\big(\|\nabla H\|_{L^2}^4+ \|\nabla u\|_{L^2}^4\big)\|\curl^2 H\|_{L^2}^2+ C\sigma\big(\|\sqrt{\rho}\dot u\|_{L^2}^2+ \|\curl^2 H\|_{L^2}^2\big)\|\nabla u\|_{L^2}^2 \notag\\
&\quad + C\sigma\|\nabla H\|_{L^2}^2\|\nabla u\|_{L^2}^6+ C\sigma\|\nabla u\|^2\big(1+\|\nabla u\|_{L^2}^2+ \|\nabla u\|_{L^2}^4+ \|\nabla H\|_{L^2}^2 +\|\nabla H\|_{L^2}^4+ \|\nabla H\|_{L^2}^6\big)\notag\\
&\quad + C\sigma'\big(\|\sqrt{\rho}\dot u\|_{L^2}^2+ \|\curl^2 H\|_{L^2}^2+ \|\nabla u\|_{L^2}^2+ \|\nabla H\|_{L^2}^4+ \|\nabla u\|_{L^2}^4+ \|\nabla H\|_{L^2}^2\|\nabla u\|_{L^2}^2\big).
\end{align}
Integrating \eqref{qwe} over $(0,T)$ gives that
\begin{align}\label{3.79}
\sup_{0 \leq t \leq T}\big(\sigma\|\sqrt{\rho}\dot u\|_{L^2}^2\big)
+\int_{0}^{T}\sigma\|\nabla\dot u\|_{L^2}^2 dt \leq C+ C\int_{0}^{T}\sigma\|\nabla u\|_{L^4}^4 dt+ \f12\int_{0}^{T}\sigma\|\nabla\dot u\|_{L^2}^2 dt.
\end{align}
We need to estimate the last two terms on the right-hand side of \eqref{3.79}.
It follows from Lemmas \ref{lem:2.1}, \ref{lem:2.6}, \eqref{3.4}, \eqref{3.7}, \eqref{3.19}, \eqref{3.29}, \eqref{3.68}, \eqref{3.75}, and \eqref{3.76} that
\begin{align*}
C\int_{0}^{T}\sigma\|\nabla u\|_{L^4}^4 dx &\leq C\int_{0}^{T}\sigma\|\nabla u\|_{L^2}\|\nabla u\|_{L^6}^3 dt \notag\\
&\leq C\int_{0}^{T}\sigma\|\nabla u\|_{L^2}\Big(\|\sqrt{\rho}\dot u\|_{L^2}^3+ \|\rho-\rho_s\|_{L^6}^3+ \|H\|_{L^{12}}^6+ \|H\cdot\nabla H\|_{L^2}^3\Big)dt \notag\\
&\leq C\int_{0}^{T}\sigma\|\nabla u\|_{L^2}\Big(\|\sqrt{\rho}\dot u\|_{L^2}^3+ \|\rho-\rho_s\|_{L^2}^3+ \|\nabla H\|_{L^2}^6+ \|\nabla H\|_{L^2}^{\f92}\|\curl^2 H\|_{L^2}^{\f32}\Big)dt \notag\\
&\leq C\sup_{0 \leq t \leq T}\big(\sigma\|\sqrt{\rho}\dot u\|_{L^2}^2\big)^{\f12}\int_{0}^{T}\|\sqrt{\rho}\dot u\|_{L^2}^2 dt \notag\\
&\quad + C\int_{0}^{T}\Big(\|\nabla u\|_{L^2}^2+ \|\nabla u\|_{L^2}^4+ \|\curl^2 H\|_{L^2}^2+ \|\nabla H\|_{L^2}^6\|\nabla u\|_{L^2}\Big)dt \notag\\
&\leq C+ C\sup_{0 \leq t \leq T}\big(\sigma\|\sqrt{\rho}\dot u\|_{L^2}^2\big)^{\f12} \leq C+ \f12\sup_{0 \leq t \leq T }\big(\sigma\|\sqrt{\rho}\dot u\|_{L^2}^2\big),
\end{align*}
which together with \eqref{3.79} implies that
\begin{align}\label{3.81}
\sup_{0 \leq t \leq T}\big(\sigma\|\sqrt{\rho}\dot u\|_{L^2}^2\big)
+\int_{0}^{T}\sigma\big(\|\nabla\dot u\|_{L^2}^2+ \|\nabla u\|_{L^4}^4\big)dt \leq C.
\end{align}
With the help of \eqref{3.76} and \eqref{3.81}, similar to the derivation of \eqref{3.59}, one can obtain the estimates of magnetic field $H$ stated in \eqref{3.77}.
\end{proof}

Now we show the uniform upper bound of density.
\begin{lemma}\label{lem:3.8}
Suppose that the conditions of Proposition 3.1 hold. Then there exists a positive constant $\varepsilon_{5}$ depending on $\mu$, $\lambda$, $\nu$, $\gamma$, $A$, $\hat\rho$, $\inf\limits_{\Omega}\phi$, $\|\phi\|_{H^2}$, $\Omega$, $M_1$, and $M_2$ such that
\begin{align}\label{cz2}
\sup_{0 \leq t \leq T}\|\rho(t)\|_{L^\infty} \leq \frac{7\hat\rho}{4},
\end{align}
provided $C_0 \leq \varepsilon_{5}$.
\end{lemma}
\begin{proof}
Denote by
\begin{align*}
D_t\rho \triangleq \rho_t + u\cdot\nabla\rho ,\quad g(\rho)\triangleq -\frac{\rho(P- P(\rho_s))}{2\mu+ \lambda}, \quad b(t)\triangleq -\frac{1}{2\mu+ \lambda}\int_{0}^{t}\rho\bigg(\f12|H|^2+ F\bigg) dt,
\end{align*}
then $\eqref{a1}_1$ can be rewritten as
\begin{align*}
D_t\rho= g(\rho)+ b'(t).
\end{align*}
In order to apply Lemma \ref{lem:2.8}, we need to deal with $b(t)$. On the one hand, it follows from \eqref{3.7}, \eqref{3.11}, and \eqref{3.75} that, for $0 \leq t \leq \sigma(T)$,
\begin{align*}
|b(t_2)- b(t_1)| &\leq C\int_{0}^{\sigma(T)}\big(\|\rho F\|_{L^\infty}+ \|\rho H\|_{L^\infty}^2\big) dt \notag\\
&\leq C\int_{0}^{\sigma(T)}\Big(\|\nabla H\|_{L^2}\|\curl^2 H\|_{L^2}+ \|\nabla H\|_{L^2}^2+ \|F\|_{L^2}^{\f14}\|\nabla F\|_{L^6}^{\f34}+\|F\|_{L^2}\Big) dt\notag\\
&\leq CC_0^{\f14}+ CC_0^{\frac{1}{16}}\int_{0}^{\sigma(T)}\Big(1+\sigma^{-\f18}\Big)\Big(\|\rho\dot u\|_{L^6}^{\f34} +\|\rho- \rho_s\|_{L^\infty}^{\f34}+ \|H\cdot\nabla H\|_{L^6}^{\f34}\Big) dt\notag\\
&\leq CC_0^{\frac{1}{16}}+ CC_0^{\frac{1}{16}}\int_{0}^{\sigma(T)}\Big(1+\sigma^{-\f18}\Big)\Big(\|\nabla\dot u\|_{L^2}^{\f34}+ \|\curl^2 H\|_{L^2}^{\f98}\Big)dt\notag\\
&\leq CC_0^{\frac{1}{16}}+ CC_0^{\frac{1}{16}}\bigg(\int_{0}^{\sigma(T)}\Big(1+ \sigma^{-\f45}\Big) dt\bigg)^{\f58}\bigg(\int_{0}^{\sigma(T)}\sigma\|\nabla\dot u\|_{L^2}^2 dt\bigg)^{\f38}\notag\\
&\quad+ CC_0^{\frac{1}{16}}\bigg(\int_{0}^{\sigma(T)}\Big(1+ \sigma^{-\f27}\Big) dt\bigg)^{\frac{7}{16}}\bigg(\int_{0}^{\sigma(T)}\|\curl^2 H\|_{L^2}^2 dt\bigg)^{\frac{9}{16}}\notag\\
&\leq CC_0^{\frac{1}{16}},
\end{align*}
where we have used
\begin{align*}
\|F\|_{L^2} &\leq C\Big(\|\nabla u\|_{L^2}+ \|\rho- \rho_s\|_{L^2}+ \|H\|_{L^2}^{\f12}\|\nabla H\|_{L^2}^{\f32}\Big)\\
&\leq CC_0^{\f14}+ C\sigma^{-\f12}\big(\sigma\|\nabla u\|_{L^2}^2\big)^{\f12} \leq CC_0^{\f14}\left(1+ \sigma^{-\f12}\right).
\end{align*}
So, for \(t \in [0, \sigma(T)]\), we can choose $N_0 = C C_0^{\frac{1}{16}}, \ N_1 = 0$, and  $\xi^* = \hat{\rho}$ in \eqref{2.31}.
Noting that
\begin{align*}
g(\xi) = -\frac{A\xi}{2\mu + \lambda} \left(\xi^\gamma - \rho_{s}^\gamma\right) \leq -N_1 = 0 \quad \text{for all } \xi \geq \xi^* = \hat{\rho},
\end{align*}
we thus deduce from \eqref{2.31} that
\begin{equation}\label{cz}
\sup_{0 \leq t \leq \sigma(T)} \|\rho(t)\|_{L^\infty} \leq \hat{\rho} + C C_0^{\frac{1}{16}} \leq \frac{3}{2}\hat{\rho},
\end{equation}
provided that
\begin{align*}
C_0 \leq \varepsilon_{4,1} \triangleq \min\left\{\varepsilon_{4}, \Big(\frac{\hat{\rho}}{2C}\Big)^{16}\right\}.
\end{align*}
\par
On the other hand, for $t\in [\sigma(T), T]$, it follows from \eqref{2.12}, \eqref{2.26}, \eqref{3.4}, and \eqref{3.64} that
\begin{align}\label{3.88}
|b(t_2)- b(t_1)| &\leq C\int_{t_1}^{t_2}\big(\|\rho F\|_{L^\infty}+ \|\rho H\|_{L^\infty}^2\big) dt \notag\\
&\leq C\int_{t_1}^{t_2}\Big(\|H\|_{L^2}^{\f12}\|\nabla H\|_{L^6}^{\f32}+\|F\|_{L^\infty}\Big) dt \notag\\
&\leq \frac{A}{2\mu+ \lambda}(t_2- t_1)+C\int_{t_1}^{t_2}\Big(\|H\|_{L^2}^2+ \|\curl^2 H\|_{L^2}^2+ \|\nabla H\|_{L^2}^2 +\|F\|_{L^\infty}^\f83\Big) dt\notag\\
&\leq \frac{A}{2\mu+ \lambda}(t_2- t_1)+C\int_{t_1}^{t_2}\Big( \|\curl^2 H\|_{L^2}^2+ \|\nabla H\|_{L^2}^2 +\|F\|_{L^4}^\f23\|\nabla F\|_{L^4}^2+\|F\|_{L^2}^\f83\Big) dt\notag\\
&\leq \frac{A}{2\mu+ \lambda}(t_2- t_1)+ CC_0^{\frac{1}{16}}+ C\int_{t_1}^{t_2}\big(\sigma^2\|\nabla\dot u\|_{L^2}^2 + \sigma\|\curl^2 H\|_{L^2}^2\big)dt \notag\\
&\leq \frac{A}{2\mu+ \lambda}(t_2- t_1)+ CC_0^{\frac{1}{16}}. \
\end{align}
Now we choose \(N_{0}=C C_{0}^{\frac{1}{16}}\), \(N_{1}=\frac{A}{\lambda+2 \mu}\), and \(\xi^*=\frac{3 \hat{\rho}}{2}\). Noting that for all \(\xi \geq \xi^*=\frac{3 \hat{\rho}}{2}>\bar{\rho}+1\),
\begin{align*}
g(\xi)=-\frac{A \xi}{2 \mu+\lambda}\left(\xi^{\gamma}-\rho_{s}^{\gamma}\right) \leq-\frac{A}{\lambda+2 \mu}=-N_{1},
\end{align*}
we then infer from \eqref{3.88} and Lemma \ref{lem:2.8} that
\begin{align}\label{cz3}
\sup_{\sigma(T) \leq t \leq T}\|\rho(t)\|_{L^{\infty}} \leq \frac{3}{2}\hat\rho+ CC_{0}^{\frac{1}{16}} \leq \frac{7}{4}\hat\rho,
\end{align}
provided that
\begin{align*}
C_0 \leq \varepsilon_{5} \triangleq \min\left\{\varepsilon_{4,1}, \Big(\frac{\hat\rho}{4C}\Big)^{16}\right\}.
\end{align*}
As a consequence, the desired \eqref{cz2} follows from \eqref{cz} and \eqref{cz3}.
\end{proof}

\section{Proof of Theorem 1.1}\label{sec4}
\par In this section we use the {\it a priori} estimates showed in Section \ref{sec3} to complete the proof of Theorem \ref{thm:main}.

\begin{proof}[{\it Proof of Theorem \ref{thm:main}}]
Let \( \eta_{\delta}(x) \) be a standard mollifier with width \( \delta \). We define
\[
\rho^{\delta}_{0} = [\eta_{\delta}*(\rho_{0}1_{\Omega})]1_{\Omega}+\delta, \quad \rho_s^{\delta}= [\eta_{\delta}*(\rho_s  1_{\Omega})]1_{\Omega}, \quad \phi^{\delta}= [\eta_{\delta}* (\phi 1_{\Omega})]1_{\Omega},
\]
meanwhile, $u_0^{\delta}$ and $H_0^{\delta}$ are the unique smooth solutions to the following elliptic equations
\begin{align}
\begin{cases}
\Delta u_0^{\delta}=\Delta(\eta^{\delta}*u_0), &x\in \Omega,\\
u_0^{\delta}\cdot n=0, \ \curl u_0^{\delta}\times n=0, &x \in \partial\Omega.
\end{cases}\ \ \ \ \ \
\begin{cases}
\Delta H_0^{\delta}=\Delta(\eta^{\delta}*H_0), &x\in \Omega,\\
H_0^{\delta}\cdot n=0, \ \curl H_0^{\delta}\times n=0, &x \in \partial\Omega.
\end{cases}
\end{align}
Using an argument similar to \cite{chen}, we can obtain a global smooth solution \( (\rho^{\delta}, u^{\delta}, H^{\delta}) \) of \eqref{a1}--\eqref{a5} with the initial data \( (\rho^{\delta}_{0}, u^{\delta}_{0}, H^{\delta}_{0}) \) for all \( t > 0 \) uniformly in \( \delta \).

Fix $x\in\Omega$ and let $B_{R}$ be a ball of radius R centered at $x$. In view of Lemma \ref{lem:2.4}, \eqref{2.12}, and the Sobolev embedding theorem, we have
\begin{align}\label{4.1}
\langle u^{\delta}(\cdot, t) \rangle^{1/2} &\leq C \big(1 + \|\nabla u^{\delta}\|_{L^{6}}\big) \notag\\
&\leq C \big(1 + \|F^{\delta}\|_{L^{6}} + \|\omega^{\delta}\|_{L^{6}} + \|P^{\delta} - P(\rho_s^{\delta})\|_{L^{6}} + \|H^{\delta}\|_{L^{12}}^{2}\big) \notag\\
&\leq C\big(1+ \|\rho^{\delta}\dot u^{\delta}\|_{L^2}+ \|H^{\delta}\cdot\nabla H^{\delta}\|_{L^2}+ \|\nabla u^{\delta}\|_{L^2}+ \|\rho^{\delta}- \rho_s^{\delta}\|_{L^6} \notag\\
&\quad + \|\rho^{\delta}- \rho_s^{\delta}\|_{L^3}+ \|H^{\delta}\|_{L^4}^2+ \|\nabla u^{\delta}\|_{L^6}+ \|\rho^{\delta}- \rho_s^{\delta}\|_{L^6}^2+ \|H^{\delta}\|_{L^{12}}^2\big) \notag\\
&\leq C(\tau), \quad t \geqslant \tau > 0.
\end{align}
 Noting that
\begin{align}\label{4.2}
\left|u^{\delta}(x, t) - \frac{1}{|B_{R} \cap \Omega|} \int_{B_{R} \cap \Omega} u^{\delta}(y, t) \, dy\right|&= \left|\frac{1}{|B_{R}\cap\Omega|}\int_{B_{R}\cap\Omega}(u^{\delta}(x, t)- u^{\delta}(y, t)) dy\right| \notag\\
&\leq \frac{1}{|B_{R}\cap\Omega|}C(\tau)\int_{B_{R}\cap\Omega}|x- y|^{\f12} dy \notag\\
&\leq C(\tau) R^{\frac{1}{2}},
\end{align}
one deduces that, for $    0 < \tau \leq t_1 \leq t_2 < \infty$,
\begin{align}\label{4.3}
|u^{\delta}(x, t_{2}) - u^{\delta}(x, t_{1})| &\leq \frac{1}{|B_{R}\cap\Omega|} \int_{t_{1}}^{t_{2}} \int_{B_{R}\cap\Omega} |u^{\delta}_{t}(y, t)| \, dy \, dt + C(\tau) R^{\frac{1}{2}} \notag\\
&\leq C R^{-3/2} |t_{2} - t_{1}|^{1/2} \left( \int_{t_{1}}^{t_{2}} \int_{B_{R}\cap\Omega} |u^{\delta}_{t}(y, t)|^{2} \, dy \, dt \right)^{\frac{1}{2}} + C(\tau) R^{\frac{1}{2}} \notag\\
&\leq C R^{-3/2} |t_{2} - t_{1}|^{1/2} \left( \int_{t_{1}}^{t_{2}} \int_{B_{R}\cap\Omega} \left(|\dot{u}^{\delta}|^{2} + |u^{\delta}|^{2}|\nabla u^{\delta}|^{2}\right) \, dy \, dt \right)^{\frac12} + C(\tau) R^{\frac12}.
\end{align}
For any \( 0 < \tau \leq t_{1} < t_{2} < \infty \), we have
\begin{align*}
\int_{t_{1}}^{t_{2}} \int |\dot{u}^{\delta}|^{2} dx dt &\leq C \int_{t_{1}}^{t_{2}} \int \big(\rho^{\delta} |\dot{u}^{\delta}|^{2} + |\rho^{\delta} - \rho_s^{\delta}|^{2} |\dot{u}^{\delta}|^{2}\big) dx dt \notag\\
&\leq C(\tau ) + C \int_{t_{1}}^{t_{2}} \|\nabla \dot{u}^{\delta}\|_{L^{2}}^{2} \|\rho^{\delta} - \rho_s^{\delta}\|_{L^{3}}^{2} dt \notag\\
&\leq C(\tau ),
\end{align*}
and
\begin{align}
\int_{t_{1}}^{t_{2}} \int |u^{\delta}|^{2} |\nabla u^{\delta}|^{2} \, dx \, dt &\leq C \sup_{t_1 \leq t \leq t_2} \|u^{\delta}\|_{L^{\infty}}^{2} \int_{t_{1}}^{t_{2}} \|\nabla u^{\delta}\|_{L^{2}}^{2} \, dt \notag\\
&\leq C \sup_{t_1 \leq t \leq t_2}\|u^{\delta}\|_{L^2}^{\f12}\|\nabla u^{\delta}\|_{L^6}^{\f32}\int_{t_1}^{t_2}\|\nabla u^{\delta}\|_{L^2}^2 dt\notag\\
&\leq C(\tau ).\notag
\end{align}
Thus, we infer from \eqref{4.3} that
\[
|u^{\delta}(x, t_{2}) - u^{\delta}(x, t_{1})| \leq C(\tau) R^{-\frac32} |t_{2} - t_{1}|^{\frac12} + C(\tau) R^{\frac12}.
\]
Choosing \( R = |t_{2} - t_{1}|^{\frac14} \), we obtain that
\begin{equation}
|u^{\delta}(x, t_{2}) - u^{\delta}(x, t_{1})| \leq C(\tau) |t_{2} - t_{1}|^{\frac18}, \quad 0 < \tau \leq t_{1} < t_{2} < \infty. \label{4.4}
\end{equation}

The estimates \eqref{4.1}--\eqref{4.4} also hold for the magnetic field \( H^{\delta} \). Hence, \( \{u^{\delta}\} \) and \( \{H^{\delta}\} \) are uniformly H{\"o}lder continuous away from \( t = 0 \). By the Ascoli--Arzel$\grave{a}$ theorem, there is a subsequence $\delta_k\rightarrow0$ satisfying
\begin{equation}
u^{\delta_k} \to u, \quad H^{\delta_k} \to H \quad \text{uniformly on compact sets in } \Omega \times (0, \infty).
\end{equation}
Moreover, similar to the argument in \cite{PL}, we can extract a further subsequence $\delta_{k'}\rightarrow0$ such that
\begin{equation}
\rho^{\delta_{k'}} \to \rho \quad \text{strongly in } L^{p}(\Omega \times (0, \infty)), \quad \forall p \in [2, \infty).
\end{equation}

Passing to the limit as \( \delta_{k'} \to 0 \), we obtain the limiting function \( (\rho, u, H) \), which is a global weak solution of \eqref{a1}--\eqref{a5} in the sense of Definition \ref{def:weak-solution} and satisfies \eqref{3.4}. The large-time behavior of $(\rho, u, H)$ as given in \eqref{1.21} is a direct consequence of the uniform bounds established in Section \ref{sec3}, and its proof follows a similar line of reasoning to that in \cite{huang}. The proof of Theorem \ref{thm:main} is thus complete.
\end{proof}

\section*{Conflict of interest}
The authors have no conflicts to disclose.

\section*{Data availability}
No data was used for the research described in the article.


\end{document}